\documentclass[pdflatex,sn-mathphys-num]{sn-jnl}

\usepackage{mathrsfs}
\usepackage{graphicx}%
\usepackage{multirow}%
\usepackage{amsmath,amssymb,amsfonts}%
\usepackage{amsthm}%
\usepackage{mathrsfs}%
\usepackage[title]{appendix}%
\usepackage{xcolor}%
\usepackage{textcomp}%
\usepackage{manyfoot}%
\usepackage{booktabs}%
\usepackage{algorithm}%
\usepackage{algorithmicx}%
\usepackage{algpseudocode}%
\usepackage{listings}%
\usepackage{tikz-cd}
\usepackage{mathtools}
\newtheorem{corollary}{Corollary}
\newtheorem{lemma}{Lemma}

\newcommand{\A}{\mathbf{A}}
\newcommand{\R}[2]{\mathbf{R}(#1,#2)}
\newcommand{\RH}[1]{\mathbf{R}(#1)}
\newcommand{\He}{\mathbf{H}}

\newcommand{\MH}{\mathbb{MH}}
\newcommand{\MN}{\mathbb{MN}}

\newcommand{\var}[1]{\mathsf{#1}}

\newcommand{\To}{\Rightarrow}
\newcommand{\RN}[2]{R(#1,#2)}
\newcommand{\Id}{\operatorname{Id}}

\newcommand{\X}[1]{\operatorname{X}(#1)}
\newcommand{\da}[1]{\operatorname{\downarrow}\hspace{-0.5mm}{#1}}
\newcommand{\ua}[1]{\operatorname{\uparrow}\hspace{-0.5mm}{#1}}
\newcommand{\uaY}[1]{\operatorname{\uparrow_Y}\hspace{-0.5mm}{#1}}
\newcommand{\uaX}[1]{\operatorname{\uparrow_X}\hspace{-0.5mm}{#1}}

\theoremstyle{thmstyleone}%
\newtheorem{theorem}{Theorem}
\newtheorem{proposition}[theorem]{Proposition}%

\theoremstyle{thmstyletwo}%
\newtheorem{example}{Example}%
\newtheorem{remark}{Remark}%

\theoremstyle{thmstylethree}%
\newtheorem{definition}{Definition}%

\begin{document}

\title[Twist-structures isomorphic to modal Nelson lattices]{Twist-structures isomorphic to modal Nelson lattices}

\author*[1,2]{\fnm{Paula} \sur{Mench\'on}}\email{paulamenchon@gmail.com}

\author[3,4]{\fnm{Ricardo O.} \sur{Rodriguez}}\email{ricardo@dc.uba.ar}

\affil*[1]{\orgdiv{Department of Logic}, \orgname{Nicolaus Copernicus University in Toru\'n}, \orgaddress{\country{Poland}}}

\affil[2]{\orgname{CONICET, Universidad Nacional del Centro de la Provincia de Buenos Aires}, \orgaddress{\city{Tandil}, \country{Argentina}}}

\affil[3]{\orgname{UBA-FCEyN. Computer Science Department}, \orgaddress{\country{Argentina}}}

\affil[4]{\orgname{UBA-CONICET. Computer Science Institute}, \orgaddress{\country{Argentina}}}


\abstract{In this paper, we introduce a new variety of Heyting algebras with two unary modal operators that are not interdefinable but satisfy the weakest condition necessary to define modal operators on Nelson lattices. To achieve this, we utilize the representation of Nelson lattices as twist structures over Heyting algebras and establish a categorical equivalence. Finally, we develop a topological duality for this new variety and apply it to derive a topological duality for modal Nelson lattices.}




\maketitle

\section{Introduction}\label{sec1}

It is well known that non-classical logics develop in two main directions: logics with additional operators (such as modal logics, temporal logics, and epistemic logics) and logics with non-classical implications (essentially, those in which $\varphi \to \psi$ is not equivalent to $\neg \varphi \vee \psi$). Our goal is to combine both approaches by taking an underlying logic with a non-classical implication and extending it with modal operators. In particular, since we would like the modal notions of necessity and possibility to be interdefinable, we will consider an underlying non-classical logic that accepts the law of double negation ($\neg \neg \varphi \to \varphi$) while rejecting the law of excluded middle ($\varphi \vee \neg \varphi$). This corresponds to constructive logic with strong negation, originally suggested by D. Nelson~\cite{Nelson49} and A. A. Markov~\cite{Markov50}. From the point of view of this kind of constructive logic, refuting a sentence $\varphi$ by reduction ad absurdum is not equivalent to refuting it by constructing a counterexample of $\varphi$. The first to formalize this idea was Vorob'ev, who provided an axiomatization for this logic~\cite{Vorobev52}. However, these axioms do not make the connection between strong negation and refutations by means of counterexamples immediately evident. Nevertheless, Vakarelov, in~\cite{Vak77}, introduced a special class of lattices whose elements and operations can be interpreted in a way very close to the intuitive notion of a counterexample. This special class of lattices is given by a twist construction in the sense of Kalman~\cite{Kalman58}. Currently, it is known as N3-lattices, and it was also studied independently by Fidel in \cite{Fidel78}. Moreover, a full representation— in fact, a categorical equivalence—was finally obtained by Sendlewski in \cite{Sendlewski}. We follow this line of research and work on one of the most challenging research trends in non-classical logic: the attempt to combine different non-classical approaches, in our case, Nelson logic and modal logic. This combination provides the ability to handle modal notions such as belief, knowledge, and obligations while interacting with other aspects of reasoning that are best addressed using many-valued logics, such as vagueness, incompleteness, and uncertainty. The study we introduce could be particularly relevant from the perspective of Theoretical Computer Science and Artificial Intelligence.

In order to achieve our goal, we first introduce an extension of the well-known twist structure construction for Nelson residuated lattices to the modal setting. The significance of this construction has grown in recent years in the study of algebras related to non-classical logics (see \cite{Busaniche, Galatos2004, RiviOno2014}). Unlike previous approaches in the literature, our proposed extension does not impose monotonicity with respect to modal operators (see \cite{JanRivi2014}).

This article is organized as follows. In Section \ref{sec2}, we introduce the fundamental concepts of the twist construction for Nelson lattices. We also present a detailed review of the known results regarding this structure, which will serve as the foundation for the subsequent sections of the paper. In Section \ref{sec3}, we introduce the variety of modal Heyting algebras, which corresponds to the algebraic semantics of the Hilbert system $\textbf{IE}_3$ presented in \cite[Section 4.2]{Oli2020}. These algebras contain two modal operators that are not interdefinable. In Section \ref{sec4}, we introduce the class of modal Nelson lattices and show that every modal Nelson lattice can be represented as a twist product over a modal Heyting algebra. We then extend this result to establish a categorical equivalence. Using this representation, in Section \ref{sec5}, we study some important subvarieties of modal Nelson lattices. Finally, in Section \ref{sec6}, we establish a topological duality for modal Heyting algebras and use it to derive a corresponding topological duality for modal Nelson lattices.

\section{Preliminaries}\label{sec2}
An algebra $\A = \langle A, \wedge, \vee, *, \To, \top, \bot \rangle$ is called a \emph{residuated lattice} if and only if the following conditions hold:
\begin{enumerate}
    \item The reduct $\langle A, \wedge, \vee, \top, \bot \rangle$ is a bounded lattice with a maximum element $\top$ and a minimum element $\bot$ (with the order denoted by $\leq$).
    \item The reduct $\langle A, *, \top \rangle$ forms a commutative monoid.
    \item The \emph{fusion} operation $*$ (sometimes referred to as the \emph{intensional conjunction} or \emph{strong conjunction}) is residuated, with $\To$ being its residual. That is, for all $a, b, c \in A$, the following condition holds:
    \begin{equation}
    a * b \leq c \iff b\leq a \To c.
    \label{res}
    \end{equation}
\end{enumerate}
In the literature, these lattices are also well-known under other names, such as integral commutative residuated monoids and $\sf{FL_{ew}}$-algebras~\cite{Hoh95, JipTsi02, Ono03}. It is worth pointing out that the class of residuated lattices, denoted $\var{RL}$, is a variety. 

If $a$ is an element of a residuated lattice $\A$, we define $a^1 = a$ and for each $n \geq 1$, $a^{n+1} = a^n * a$. A derived unary operation $\sim$ is defined by $\sim a = a \To \bot$. As usual, this operation is called the negation operation, and an element $a$ satisfying $a = \sim a$ is called a negation fixed point. A residuated lattice is called involutive if it satisfies the double negation equation:
$$ a = \sim \sim a. $$

A \emph{Nelson residuated lattice}, or simply \emph{Nelson lattice} (N3), is an involutive residuated lattice satisfying:
\[
((a^2 \To b) \wedge ((\sim b)^2 \To \sim a)) \To (a \To b) = \top.
\]

In \cite[Theorem 2.2]{Busaniche}, it is proved that every Nelson lattice $\mathbf{A}$ satisfies \emph{3-potency}, i.e., for all $a \in A$,  $a^3 = a^2$. As a consequence of this result, the following corollary can be obtained: 
\begin{corollary} \label{aux1}
Let $\mathbf{A}$ be a Nelson lattice. Then, for all $a, b \in A$:  
 \begin{itemize}
     \item[(a)] If $a^2 \To b = \top$, then $a^2 \To b^2 = \top$.
     \item[(b)] $(a * b)^2 = (a \wedge b)^2$.  
 \end{itemize}
 \end{corollary}
This result will become important in our presentation.

\subsection{Representation of Nelson lattices as twist-structures over Heyting algebras}

A \emph{Heyting algebra} is a  residuated lattice
$\He =\langle H,\wedge,\vee,\rightharpoonup,\top,\bot\rangle$  in which the monoidal operation coincides with the meet, i.e., $a * b = a \wedge b$. A non-empty subset $F \subseteq H$ is called a \emph{filter} if it is upward closed and closed under finite meets. In a Heyting algebra $\He$, filters and congruences are in bijective correspondence. For each congruence $\theta$ on $\He$, the set $F_\theta = \{a \in H : (a,1) \in \theta\}$ is a filter of $\He$. Conversely, for each filter $F \subseteq H$, the relation $\theta_F = \{(a,b) \in H \times H : (a \to b) \wedge (b \to a) \in F\}$ is a congruence on $\He$. Moreover, these assignments are mutually inverse, thus yielding an isomorphism between the lattice of filters and the lattice of congruences of $\He$. A filter $F$ of $\He$ is said to be \emph{Boolean} if the quotient $\He/\theta_F$ is a Boolean algebra.

Let $\He=\langle H,\wedge,\vee,\rightharpoonup,\top,\bot\rangle$ be a Heyting algebra, and define the pseudocomplement of $a \in H$ by $-a := a \rightharpoonup \bot$.
An element $a \in H$ is said to be \emph{dense} if $--a = \top$, or equivalently, if $-a = \bot$. We denote by $D(\He)$ the filter of dense elements of $\He$, that is,  $D(\He) = \{a \in H : -a =\bot\}$.  It is well known and easy to verify that a filter $F$ of $\He$ is Boolean if and only if $D(\He) \subseteq F$. The interested reader can find more details on Heyting algebras in \cite{Ruth65}. 

Now, we introduce the usual connection between Heyting algebras and Nelson residuated lattices.

\begin{theorem}[Sendlewski + Theorem 3.1 in \cite{Busaniche}] \label{TwistRepresentation} 
Given a Heyting algebra $\He$ 
and a boolean filter $F$ of $\He$, define:
\[
\RN{\He}{F}:=\{(x,y)\in H\times H \mid x\wedge y=\bot \text{ and } x\vee y\in F\}.
\] 
Then, the following hold:
\begin{enumerate}
    \item $\R{\He}{F}=\langle \RN{\He}{F},\wedge,\vee, *, \To,\bot,\top\rangle$ forms a Nelson lattice, where operations are defined as:
    \begin{itemize}
        \item $(x,y)\vee(s,t)=(x\vee s,y\wedge t)$, 
        \item $(x,y)\wedge(s,t)=(x\wedge s,y\vee t)$,
        \item $(x,y)*(s,t)=(x\wedge s,(x \rightharpoonup t)\wedge (s\rightharpoonup y))$,
        \item $(x,y)\To (s,t)=((x\rightharpoonup s)\wedge(t\rightharpoonup y),x\wedge t)$, 
        \item $\top=(\top,\bot)$ and $\bot=(\bot,\top)$.
    \end{itemize}
    
    \item The negation operations are given by:
    \begin{itemize}
        \item $\sim (x,y)=(y,x)$ (strong negation),
        \item $\neg (x,y) = (-x, x)$ (weak negation).
    \end{itemize}

    \item For every Nelson lattice $\A$, there exists a unique (up to isomorphism) Heyting algebra $\He_\A$ and a unique boolean filter $F_\A$ of $\He_\A$ such that $\A$ is isomorphic to $\R{\He_\A}{F_\A}$.

    \item If $F_1,F_2$ are boolean filters of $\He$, then $\R{\He}{F_1}$ is a subalgebra of $\R{\He}{F_2}$ if and only if $F_1\subseteq F_2$.

    \item If $\mathcal{V}$ is a variety of Nelson lattices, then the class $\mathcal{H}^\mathcal{V} :=\{\He_\A \mid \A\in\mathcal{V}\}$ forms a variety of Heyting algebras.
\end{enumerate}
\end{theorem}

 From the previous theorem, it follows immediately that:
 \[
    (x,y)^2=(x,y)*(x,y)=(x,x\rightharpoonup y)=(x,-x),
\]
since $x\wedge y=\bot$ implies $x\rightharpoonup y=-x$.

For convenience, we will omit the subscript in $\He_\A$ whenever it is clear from the context. Furthermore, we will use the notation $\RH{\He}$ instead of $\R{\He}{H}$. With this convention, the well-known result by Fidel and Vakarelov can be derived as a corollary of the previous theorem:

\begin{corollary}
For every Nelson lattice $\A$, there exists a Heyting algebra $\He$ such that $\A$ is isomorphic to a subalgebra of $\RH{\He}$.
\end{corollary}

On each Nelson lattice $\A = \langle A, \vee, \wedge, *, \To, \top, \bot \rangle$, define the binary operation $\to$ by the prescription $x \to y := x^2 \To y$, and let $\A' = \langle A, \vee, \wedge, \to, \top, \bot \rangle$. For each Nelson lattice $\A$, the binary relation $\equiv$ defined on $\A$ by the prescription 
\begin{equation} \label{equivalencia}
x \equiv y \text{ if and only if } x^2 = y^2
\end{equation}
is a congruence on the algebra $\A'$. The quotient $\A'/\equiv$, with the natural operations, is a Heyting algebra (see \cite[Theorem 3.4]{Busaniche}). 

\begin{remark} \label{SpecialHeytingAlgebra}
Let $\A$ be a Nelson lattice. Define $H^* := \{ a \in A : a^2 = a \}$ and the operations $a \star^* b = (a \star b)^2$ for every binary lattice operation, and $a \rightharpoonup^* b = (a \to b)^2$. In \cite{Sendlewski}, Sendlewski proved that $\He^* = \langle H^*, \vee^*, \wedge^*, \rightharpoonup^*, \bot, \top \rangle$ is a Heyting algebra isomorphic to $\A'/\equiv$ (using the assignment $[a]_\equiv \mapsto a^2$). Note that $a^2$ is the least element in the class $[a]_\equiv$ with respect to the lattice ordering in $\A$. 

Furthermore, if $a \in H^*$, we have:
\[
-^* a = a \rightharpoonup^* \bot = (a \to \bot)^2 = (a^2 \To \bot)^2 = (\sim a^2)^2.
\]
Since $a^2 = a$, it follows that $-^* a = (\sim a)^2$.
\end{remark}

Let $F^* \subseteq H^*$ be the set 
\[
F^* = \{ (a \vee \sim a)^2 : a \in A \} = \{ b^2 : \sim b \leq b \}.
\]
Then $F^*$ is a boolean filter of $\He^*$ because it contains the dense elements. That is, if $a\in H^*$ and $-^* a = \bot$, then $(a \vee \sim a)^2= a$.


\begin{remark}
It is worth noticing that for any Nelson lattice $\A$, $\He_\A$ is isomorphic to $\He^*$.
\end{remark}

\begin{remark} \label{transprop}
It is also worth noticing that by using the assignment $[a]_\equiv \mapsto a^2$, we can fix some conditions on the Heyting algebra by putting the 2-potency of the same condition on the Nelson lattice. For instance, a Stonean Heyting algebra can be obtained by requiring that the Nelson algebra satisfies the equation:
\[
(\sim x^2)^2 \vee (\sim(\sim x^2)^2)^2 = \top.
\]
\end{remark}

The following lemma will be used in the proof of Theorem~\ref{MappingNelsonHeyting} and again in Section~\ref{sec4}. Although not stated in this form, it follows from results in \cite{Sendlewski}.

\begin{lemma}\label{SendlewskiBoolanResult}
Let $\A$ be a Nelson lattice, $a,b \in A$, and consider $F^*$ as defined before.
\begin{enumerate}
\item If $a^2 = b^2$ and $(\sim a)^2 = (\sim b)^2$, then $a = b$.
\item If $(a \wedge b)^2 = \bot$ and $a^2 \vee b^2 \in F^*$, then there exists $c \in A$ such that $c^2 = a^2$ and $(\sim c)^2 = b^2$.
\end{enumerate}
\end{lemma}

\begin{proof}
    1. The proof can be found in Theorem 3.6 of \cite{Busaniche}.
    
    2. Since $\bot=(a\wedge b)^2=(a*b)^2=a^2*b^2$, it follows that $a^2\leq \sim b^2$. Let $d\in A$ be such that $\sim d\leq d$ and $(a\vee b)^2=d^2$. Define $c\coloneq d\wedge \sim b^2$. Then, 
    Then
    \[
    c^2=(d\wedge \sim b^2)^2=(d^2 * \sim b^2)^2.
    \]
    Since $d^2=a^2\vee b^2$, we obtain
    \[
    c^2=((a^2\vee b^2)*\sim b^2)^2.
    \]
    By distributivity,
    \[
    ((a^2\vee b^2)*\sim b^2)^2=((a^2*\sim b^2)\vee (b^2*\sim b^2))^2=(a^2*\sim b^2)^2,
    \]
    because $b^2*\sim b^2=\bot$. Finally, as $a^2\leq \sim b^2$,
    \[
    (a^2*\sim b^2)^2=(a^2\wedge \sim b^2)^2=a^2.
    \]
    Hence $c^2=a^2$. On the other hand,
    \[
    \sim c=\sim(d\wedge \sim b^2)=\sim d\vee b^2,
    \]
    and therefore
    \[
    (\sim c)^2=(\sim d\vee b^2)^2=(\sim d)^2\vee b^2.
    \]
    Since $\sim d\leq d$, we have
    \[
    (\sim d)^2\leq d*\sim d=\bot,
    \]
    whence $(\sim d)^2=\bot$. It follows that
    \[
    (\sim c)^2=(\sim d)^2\vee b^2=\bot\vee b^2=b^2.
    \]
\end{proof}

\begin{theorem} \label{MappingNelsonHeyting}
The following statements are valid:

\begin{enumerate}\label{theorem isom}
\item \label{iso2}
Let $\A$ be a Nelson lattice. Then, $\A$ is isomorphic to $\R{\He^*}{F^*}$, where the isomorphism $e\colon A \to \RN{\He^*}{F^*}$ is defined by:
\[
e(a) := (a^2, (\sim a)^2).
\]

\item \cite{Odintsov} Let $\He$ be a Heyting algebra, and let $\mathbf{B}$ be a subalgebra of $\RH{\He}$ such that $\pi_1(B) = H$ (and similarly, $\pi_2(B) = H$), where $\pi_i$ denotes the projection onto the $i$th coordinate of the direct product. We define the following:
\[
I(\mathbf{B}) = \{a \vee \sim a : a \in B\} \quad \text{and} \quad F(\mathbf{B}) = \pi_1(I(\mathbf{B})).
\]
Then, $D(\He) \subseteq F(\mathbf{B})$ is a filter of $\He$, and  $B = \RN{\He}{F(\mathbf{B})}$.
\end{enumerate}
\end{theorem}

It is important to note that, according to \ref{iso2} of Theorem \ref{theorem isom} and Remark \ref{transprop}, if we take a valid equation in the context of Heyting algebras, its 2-potency will also be valid in Nelson lattices.

In the next sections, we are going to extend these results to a modal context.

\section{Modal Heyting algebras}\label{sec3}

In this section, we introduce a new variety of Heyting algebras with modal operators.

\begin{definition} \label{ModalHeytingAlgebra}
A \emph{modal Heyting algebra} is an algebra $\mathbf{M} = \langle \He, \square, \Diamond \rangle$ such that the reduct $\He$ is a Heyting algebra, $\square$ and $\Diamond$ are unary operators on $\He$, and for all $a, b \in H$, the following equation is satisfied:
\begin{align}
    \square a \wedge \Diamond (-a \wedge b) = \bot. \tag{mH} \label{mH}
\end{align}
We denote by $\mathcal{MH}$ the variety of these algebras.
\end{definition}

\begin{lemma} \label{equivAxrule1}
Let $\mathbf{M} = \langle \He, \square, \Diamond \rangle$ be a modal Heyting algebra. Then the equation \eqref{mH} is equivalent to the following quasi-equation:
\begin{align}
     \text{If } a \wedge b = \bot, \text{ then } \square a \wedge \Diamond b = \bot. \tag{mH'} \label{mH'}
\end{align}
\end{lemma}

\begin{proof}
$\Rightarrow$) Assume \eqref{mH}. If $a \wedge b = \bot$, then $b \leq -a$, and therefore $-a \wedge b = b$. Thus, $\square a \wedge \Diamond b = \square a \wedge \Diamond (-a \wedge b) = \bot$. \\
$\Leftarrow$) Assume \eqref{mH'}. Since $a \wedge -a \wedge b = \bot$, we conclude that $\square a \wedge \Diamond (-a \wedge b) = \bot$.
\end{proof}

To justify Condition \eqref{mH'}, we refer to \cite{Oli2020}, where the authors study different intuitionistic non-normal modal logics. Specifically, they consider a family of these logics, including both operators $\Box$ and $\Diamond$. They emphasize the challenge of finding suitable interactions between the modalities without reaching interdefinability. Their logics are distinguished by the different strengths of interactions between the modalities. The underlying logic of our variety of modal Heyting algebras corresponds to their Hilbert system $\textbf{IE}_3$ in \cite[Section 4.2]{Oli2020}, where our \eqref{mH'} quasi-equation aligns with their {\bf str} rule.

Let $\mathbf{M} = \langle \He, \square, \Diamond \rangle$ be a modal Heyting algebra. Observe that the degenerate modal case, in which $\Box a = \top$ and $\Diamond a = \bot$ for every $a \in H$, is included in Condition \eqref{mH'}.\footnote{In Kripke semantics, this corresponds to worlds with no successors.}

According to our definition, modal Heyting algebras form a variety, and, to the best of our knowledge, this has not been previously mentioned. However, there are some well-known extensions of this class. For instance, the variety of \emph{relational}\footnote{In the classical modal logic literature, this class is usually referred to as \emph{normal}. However, we reserve this terminology for a different class of algebras.} modal Heyting algebras is obtained by including the following equations:
\begin{align}
    \square \top = \top, \tag{mH1} \label{mH1} \\
    -\Diamond a = \square -a, \tag{mH2} \label{mH2} \\
    \square (a \rightharpoonup b) \rightharpoonup (\square a \rightharpoonup \square b) = \top. \tag{mH3} \label{mH3}
\end{align}

Note that \eqref{mH'} implies that $\square a \wedge \Diamond -a = \bot$ and $\square -a \wedge \Diamond a = \bot$, which leads to the conclusion that $\Diamond -a \leq -\square a$ and $\square -a \leq -\Diamond a$. Furthermore, assuming \eqref{mH1}, we get $\Diamond \bot = \bot$.

\begin{remark}
It is worth noting that in a relational modal Heyting algebra $\mathbf{M}$ (i.e., one that satisfies \eqref{mH1}-\eqref{mH3}), Condition \eqref{mH'} is derivable from \eqref{mH1}-\eqref{mH3}. Specifically, if $a \wedge b = \bot$, then we have $\square (a \wedge b \rightharpoonup \bot) = \square(a \rightharpoonup (b \rightharpoonup \bot)) = \top$ according to \eqref{mH1}. Applying equation \eqref{mH3}, we also get $\square a \rightharpoonup \square(b \rightharpoonup \bot) = \top$. Finally, by using equation \eqref{mH2}, we conclude that $\square a \wedge \Diamond b = \bot$.
\end{remark}

It is also well-known that if an algebra satisfies equation \eqref{mH1}, then \eqref{mH3} is equivalent to the equation $\square (a \wedge b) = \square a \wedge \square b$.

Returning to our definition of Modal Heyting algebras, it is important to note that the most general version typically found in the literature is a structure \( \langle \He, \square, \Diamond \rangle \) where \( \He \) is a Heyting algebra, and the operators \( \square \) and \( \Diamond \) are independent of each other. In contrast, we have introduced property \eqref{mH} in Definition \ref{ModalHeytingAlgebra} to impose a specific interaction between these modalities, which allows us to characterize modal Nelson lattices in the following section.

\section{Modal Nelson lattices}\label{sec4}

\begin{definition}  \label{ModalN3Lattice} 
A \emph{modal Nelson residuated lattice} (for short, MN-lattice) is an algebra $\mathbf{N}= \langle \A,\blacksquare \rangle $ such that the reduct $ \A $ is a Nelson lattice and, for all $ a, b \in A $, the following conditions hold:
\begin{align}
    & \text{If } a^2 = b^2, \text{ then } (\blacksquare a)^2 = (\blacksquare b)^2 \text{ and } (\sim \blacksquare \sim a)^2 = (\sim \blacksquare \sim b)^2, \tag{mN1} \label{mN2} \\
    & (\blacksquare a \wedge \sim \blacksquare (a^2 \vee \sim b))^2 = \bot. \tag{mN2} \label{mN3}
\end{align}
\end{definition}
In these algebras, one may define the dual modal operator $\blacklozenge$ by
\begin{equation}\label{mN1}\tag{mN3}
\blacklozenge a \coloneqq \sim \blacksquare \sim a.
\end{equation}

Accordingly, Condition \eqref{mN3} may be equivalently written as
\[
  (\blacksquare a \wedge \blacklozenge (\sim a^2 \wedge b))^2 = \bot.  
\]

As in the case of modal Heyting algebras, Condition \eqref{mN3} is equivalent to the following quasi-equation: 
\begin{align}
    \text{If } (a \wedge b)^2 = \bot, \text{ then } (\blacksquare a \wedge \blacklozenge b)^2 = \bot. \tag{mN2'} \label{mN3'}
\end{align}

The proof of this equivalence is similar to Lemma \ref{equivAxrule1}.

Condition \eqref{mN2} is equivalent to the following equations:
\begin{align}
    & (\blacksquare a)^2 = (\blacksquare a^2)^2, \tag{mN1$\blacksquare$'}\label{mN2s}\\
    & (\blacklozenge a)^2 = (\blacklozenge a^2)^2. \tag{mN1$\blacklozenge$'}\label{mN2d}
\end{align}

Therefore, the class of $\mathcal{MN}$-lattices forms a variety, which we denote by $\mathcal{MN}$. In particular, the modal operations are well defined on the algebra $\He^*$.

 Now, we introduce some extensions of our basic notion of MN-lattice. For instance, an MN-lattice $\mathbf{N}$ is said to be $\blacksquare$-\emph{regular} if, in addition to the conditions from Definition \ref{ModalN3Lattice}, it satisfies the following equation:
 \begin{align}
     \blacksquare(a \wedge b) = \blacksquare a \wedge \blacksquare b. \tag{mN4}\label{mN4}
 \end{align}

 Moreover, if $\mathbf{N}$ is a $\blacksquare$-regular modal N-lattice (for short, RMN-lattice), using \eqref{mN1} and \eqref{mN4}, we can conclude that it is also $\blacklozenge$-regular, i.e., it satisfies:
 \begin{align}
     \blacklozenge(a \vee b) = \blacklozenge a \vee \blacklozenge b. \tag{mN4'} \label{mN4'}
 \end{align}

 Finally, we say that a modal Nelson lattice is \emph{relational} if it is $\blacksquare$-regular and, in addition, satisfies:
 \begin{align}
      \blacksquare \top = \top. \tag{mN5}\label{mN5}
 \end{align}

\begin{remark}
    It is worth noting that if $\mathbf{N}$ is a relational modal N-lattice, then its associated Heyting algebra $\He^*$ is relational as well. In particular, conditions \eqref{mN4} and \eqref{mN5} correspond to \eqref{mH3} and \eqref{mH1}, respectively. This correspondence will become apparent in the representation theorem, where modal Nelson algebras are characterized in terms of modal Heyting algebras.
\end{remark}


 According to \eqref{mN2}, we can reproduce the following classical result on RMN-lattices:

 \begin{lemma} \label{monotonicity_RMN}
 If $\mathbf{N}=\langle \A,\blacksquare\rangle$ is a $\blacksquare$-regular modal N-lattice, then it satisfies the following monotonicity properties:
 \begin{align}
     &\text{if } a^2\leq b, \text{ then } (\blacksquare a)^2\leq \blacksquare b, \tag{mN6} \label{mN6} \\
     &\text{if } (\sim a)^2\leq \sim b, \text{ then } (\sim \blacksquare a)^2\leq \sim \blacksquare b. \tag{mN7} \label{mN7}
 \end{align}
 \end{lemma}
 \begin{proof}
 If $a^2\leq b$, then by Corollary \ref{aux1}, we know that $a^2\leq b^2$. Under this condition, $a^2 = (a\wedge b)^2$, and by equation \eqref{mN4}, we obtain:
 \[
 (\blacksquare a)^2 = (\blacksquare (a \wedge b))^2 = (\blacksquare a \wedge \blacksquare b)^2 \leq \blacksquare b.
 \]

 A similar argument applies to the second property. If $(\sim a)^2\leq \sim b$, then $(\sim b)^2 = (\sim a \vee \sim b)^2$. Using equations \eqref{mN4'} and \eqref{mN1}, we obtain:
 \[
 \sim \blacksquare b \geq (\sim \blacksquare b)^2 = (\blacklozenge \sim b)^2 = (\blacklozenge (\sim a \vee \sim b))^2 =  (\blacklozenge \sim a \vee \blacklozenge \sim b)^2 \geq (\blacklozenge \sim a)^2 = (\sim\blacksquare a ) ^2.
 \]
 \end{proof}
The last result establishes a connection between our work and Section 4 in \cite{JanRivi2014}.
 \begin{remark}
 It is worth noting that if $\mathbf{N}=\langle \A,\blacksquare\rangle$ is a $\blacksquare$-regular modal N-lattice, then equation \eqref{mN3} from Definition \ref{ModalN3Lattice} is derivable.  
 \end{remark}

 Since, according to \eqref{mN2s} and \eqref{mN2d}, the modal operators $\blacksquare$ and $\blacklozenge$ are compatible with the relation $\equiv$, it is natural to define their corresponding operations on the quotient structure $\A' / \equiv$. 


\begin{lemma}\label{lemma Heytingfilter}
Let $\mathbf{N}=\langle \A,\blacksquare\rangle$ be an MN-lattice. Then, the structure $\mathbf{M}_\mathbf{N}^*=\langle{\He}^*, \Box^*, \Diamond^*\rangle$, where ${\He}^*$ is the Heyting algebra defined in Remark \ref{SpecialHeytingAlgebra}, and the modal operators $\Box^*$ and $\Diamond^*$ are defined for every $a \in H^*$ as:
\begin{align}
    \Box^* a &= (\blacksquare a)^2, \\
    \Diamond^* a &= (\blacklozenge a)^2,
\end{align}
is a modal Heyting algebra. Furthermore, if we define the set:
\[
   F^* = \{(a \vee \sim a)^2 : a \in A \},
\]
then $F^*$ is a boolean filter satisfying the following condition for every $a, b \in H^*$:
\[
   \text{if } a \vee^* b \in F^* \text{ and } a \wedge^* b = \bot, \text{ then } \Box^* a \vee^* \Diamond^* b \in F^*.
\]
\end{lemma}
\begin{proof}
Recall from Remark \ref{SpecialHeytingAlgebra} that the reduct ${\He}^* = \langle H^*,\vee^*,\wedge^*,\rightharpoonup^*,\bot,\top\rangle$ is a Heyting algebra. The operators are well defined as an immediate consequence of Property {\eqref{mN2}} in Definition \ref{ModalN3Lattice}. 

To prove Condition {\eqref{mH'}}, which is equivalent to {\eqref{mH}} by Lemma \ref{equivAxrule1}, assume that $a \wedge^* b = \bot$ for some $a, b \in H^*$. By definition, we have: $ a \wedge^* b = (a \wedge b)^2 = \bot$.
Applying Property {\eqref{mN3'}}, we obtain $(\blacksquare a \wedge \blacklozenge b)^2 = \bot$.
Moreover, since $(\blacksquare a \wedge \blacklozenge b)^2 = ((\blacksquare a)^2 \wedge (\blacklozenge b)^2)^2$,
it follows that $\Box^* a \wedge^* \Diamond^* b = \bot$.

Now, suppose that $a,b\in H^*$, $a\vee^*b\in F^*$ and $a\wedge^*b=\bot$. By Lemma \ref{SendlewskiBoolanResult}, it follows that there exists $c\in A$ such that $c^2=a$ and $(\sim c)^2=b$. Then, $\square^*a=(\blacksquare a)^2=(\blacksquare c^2)^2$, and by \eqref{mN2s}, $\square^*a=(\blacksquare c)^2$. Similarly, $\Diamond^*b=(\blacklozenge b)^2=(\blacklozenge(\sim c)^2)^2$. By \eqref{mN2d}, $(\blacklozenge(\sim c)^2)^2=(\blacklozenge(\sim c))^2=(\sim \blacksquare c)^2$, and hence $\Diamond^*b=(\sim \blacksquare c)^2$. Therefore $\square^*a\vee \Diamond^*b=(\blacksquare c)^2\vee(\sim \blacksquare c)^2=(\blacksquare c\vee \sim \blacksquare c)^2\in F^*$.
\end{proof}

Our goal is to prove that every modal Nelson algebra can be represented as a twist-structure over a modal Heyting algebra. First, we will demonstrate that the modal operators are well-defined on the twist product and that they satisfy the equations to belong to the variety $\mathcal{MN}$.

\begin{theorem} \label{ModalHeyToModalNelson}
Let $\mathbf{M}=\langle \He,\square,\Diamond\rangle$ be a modal Heyting algebra as in Definition \ref{ModalHeytingAlgebra}. Then, $\mathbf{N}(\mathbf{M})=\langle \RH{\mathbf{H}},\blacksquare\rangle$ is a modal Nelson lattice, where the reduct $\RH{\mathbf{H}}=\langle R(\mathbf{H}),\wedge,\vee, *, \To,\bot,\top\rangle$
is defined as in Theorem \ref{TwistRepresentation}(1), and the modal operators $\blacksquare$ and $\blacklozenge$ are given by:
\begin{align}
    \blacksquare(x,y) &= (\square x, \Diamond y), \\  
    \blacklozenge(x,y) &= (\Diamond x, \square y).
\end{align}
\end{theorem}
\begin{proof}
    By Theorem \ref{TwistRepresentation}, the reduct $\langle R(\mathbf{H}),\wedge,\vee, *, \To,\bot,\top\rangle$ is a Nelson lattice. Suppose $(a,b) \in R(\mathbf{H})$. Then, by definition, we have $a \wedge b = \bot$. Applying Condition \eqref{mH'}, we obtain $\Box a \wedge \Diamond b = \bot$. Consequently, $(\Box a, \Diamond b)$ and $(\Diamond b, \Box a)$ also belong to $R(\mathbf{H})$.

    To prove that $\mathbf{N}(\mathbf{M})$ is a modal Nelson lattice, we must verify that the modal operators satisfy the three properties given in Definition \ref{ModalN3Lattice}.   

     Assume that $(a,b) \in R(\mathbf{H})$. By the definitions of the modal operators, we have $\blacklozenge(a,b) = ( \Diamond a, \Box b) = \sim (\Box b, \Diamond a) = \sim \blacksquare (b,a) = \sim\blacksquare \sim (a,b)$.

     \eqref{mN2}  To prove the first condition, assume that $(a,b)^2=(c,d)^2$. Then, $(a,b)^2=(a, a \rightharpoonup b) = (c, c \rightharpoonup d) = (c,d)^2$. By definition, $a \wedge b = c \wedge d = \bot$, which implies that $a \rightharpoonup b = - a$ and $c \rightharpoonup d = - c$. Hence, $(a,b)^2=(a, - a)$ and $(c,d)^2=(c,- c)$. This implies that $a=c$, and by applying the modal operators appropriately, we obtain $\Box a = \Box c $ and $- \Box a= - \Box c$. Therefore, $(\blacksquare(a,b))^2 = (\Box a, \Diamond b)^2 = (\Box a, - \Box a)$ due to $\Box a \wedge \Diamond b = \bot$. Similarly, $(\blacksquare(c,d))^2 = (\Box c, \Diamond d)^2 = (\Box c, - \Box c)$. Combining both observations, we obtain the desired result. Analogously, we prove that $(\blacklozenge (a,b))^2=(\blacklozenge (c,d))^2$.

    \eqref{mN3} For the second condition, we verify the equivalent Condition \eqref{mN3'}. Assume that $((a,b) \wedge (c,d))^2= (\bot,\top)$. This implies: $(a \wedge c, b \vee d)^2 = (\bot,\top)$. Since $a \wedge c = \bot$, we apply the quasi-equation \eqref{mH'} from Lemma \ref{equivAxrule1} and obtain $\Box a \wedge \Diamond c = \bot$.  Using this, we compute: $(\blacksquare(a,b) \wedge \blacklozenge(c,d))^2$, which is, by definition, equal to $(\Box a \wedge \Diamond c, (\Box a \wedge \Diamond c) \rightharpoonup (\Diamond b \vee \Box d)) = (\bot,\top)$.
   This confirms the final condition and completes the proof.
\end{proof}

\begin{example}
   
Let us consider the modal Heyting algebra $\mathbf{M} = \langle \He, \square, \Diamond \rangle$ depicted in Figure \ref{fig1*}. It is straightforward to verify that $\mathbf{M}$ satisfies Equation \eqref{mH}, as both $\square \bot = \bot$ and $\Diamond \bot = \bot$ hold.

\begin{figure}[ht]
    \centering
\begin{tikzpicture}
   
\begin{scope}
\node [
label=below:{${\bot}$}, label=below:{ }] (n5)  {} ;
\node [above of=n5, label=right:{${b}$}, label=below:{} ] (n1)  {} ;
\node [above right of=n1,label=right:{${c}$}, label=below:{ }] (n2)  {} ;
\node [above left of=n1,label=left:{${a}$}, label=below:{ }] (n3)  {} ;
\node [above right of=n3,label=above:{${\top}$}, label=below:{ }] (n4)  {} ;

\draw  (n1) -- (n2);
\draw  (n1) -- (n3);
\draw  (n2) -- (n4);
\draw  (n3) -- (n4);
\draw  (n1) -- (n5);

\draw [fill] (n1) circle [radius=.5mm];
\draw [fill] (n2) circle [radius=.5mm];
\draw [fill] (n3) circle [radius=.5mm];
\draw [fill] (n4) circle [radius=.5mm];
\draw [fill] (n5) circle [radius=.5mm];

\draw [->,dashed,thick,blue] (n4) edge[bend left] (n2);
\draw [->,dashed,thick,blue] (n5) arc [radius=3.5mm, start angle=450, end angle= 100]  (n5);
\draw [->,dashed,thick,blue] (n3) arc [radius=3.5mm, start angle=0, end angle= 340]  (n3);
\draw [->,dashed,thick,blue] (n2) edge[bend right] (n3);
\draw [->,dashed,thick,blue] (n1) edge[bend right] (n2);

\end{scope}

\begin{scope}[xshift=5cm]

\node [
label=below:{${\bot}$}, label=below:{ }] (n5)  {} ;
\node [above of=n5, label=right:{${b}$}, label=below:{} ] (n1)  {} ;
\node [above right of=n1,label=right:{${c}$}, label=below:{ }] (n2)  {} ;
\node [above left of=n1,label=left:{${a}$}, label=below:{ }] (n3)  {} ;
\node [above right of=n3,label=above:{${\top}$}, label=below:{ }] (n4)  {} ;

\draw  (n1) -- (n2);
\draw  (n1) -- (n3);
\draw  (n2) -- (n4);
\draw  (n3) -- (n4);
\draw  (n1) -- (n5);

\draw [fill] (n1) circle [radius=.5mm];
\draw [fill] (n2) circle [radius=.5mm];
\draw [fill] (n3) circle [radius=.5mm];
\draw [fill] (n4) circle [radius=.5mm];
\draw [fill] (n5) circle [radius=.5mm];

\draw [->,dashed,thick,red] (n5) arc [radius=3.5mm, start angle=450, end angle= 100]  (n5);
\draw [->,dashed,thick,red] (n3) edge[bend right] (n1);
\draw [->,dashed,thick,red] (n2) edge[bend right] (n3);
\draw [->,dashed,thick,red] (n1) edge[bend right] (n2);
\draw [->,dashed,thick,red] (n4) arc [radius=3.0mm, start angle=270, end angle= 610]  (n2);
\end{scope}
\end{tikzpicture}
\caption{The Hasse diagram of a modal Heyting algebra $\mathbf{M} = \langle \He, \square, \Diamond \rangle$. The behavior of the $\Diamond$ operator is depicted in red on the right, and the behavior of the $\square$ operator is shown in blue on the left.}
\label{fig1*}
\end{figure}

Then, the modal Nelson lattice $\mathbf{N}(\mathbf{M}) = \langle \RH{\He}, \blacksquare \rangle$ is shown in Figure \ref{fig2*}. Note that it is a centered Nelson lattice, i.e. it has a negation fixed point. In this example, since both $\square \bot = \bot$ and $\Diamond \bot = \bot$, the pair $(\bot, \bot)$ serves as a fixed point for the modal operators.

\begin{figure}[ht]
    \centering
\begin{tikzpicture}

\begin{scope}
\node [
label=below right:{${(\bot,\top)}$}, label=below:{} ] (n1)  {} ;
\node [above right of=n1,label=right:{${(\bot,c)}$}, label=below:{ }] (n2)  {} ;
\node [above left of=n1,label=left:{${(\bot,a)}$}, label=below:{ }] (n3)  {} ;
\node [above right of=n3,label=left:{${(\bot,b)}$}, label=below:{ }] (n4)  {} ;
\node [above of=n4, label=left:{${(\bot,\bot)}$}, label=below:{ }] (n5)  {} ;
\node [above of=n5, label=left:{${(b,\bot)}$}, label=below:{ }] (n6)  {} ;
\node [above right of=n6,label=right:{${(c,\bot)}$}, label=below:{ }] (n7)  {} ;
\node [above left of=n6,label=left:{${(a,\bot)}$}, label=below:{ }] (n8)  {} ;
\node [above right of=n8, label=above right:{${(\top,\bot)}$}, label=below:{ }] (n9)  {} ;
\draw  (n1) -- (n2);
\draw  (n1) -- (n3);
\draw  (n2) -- (n4);
\draw  (n3) -- (n4);
\draw  (n4) -- (n5);
 \draw  (n5) -- (n6);
\draw  (n6) -- (n7);
\draw  (n6) -- (n8);
\draw  (n7) -- (n9);
\draw  (n8) -- (n9);

\draw [fill] (n1) circle [radius=.5mm];
\draw [fill] (n2) circle [radius=.5mm];
\draw [fill] (n3) circle [radius=.5mm];
\draw [fill] (n4) circle [radius=.5mm];
\draw [fill] (n5) circle [radius=.5mm];
 \draw [fill] (n6) circle [radius=.5mm];
  \draw [fill] (n7) circle [radius=.5mm];
   \draw [fill] (n8) circle [radius=.5mm];
    \draw [fill] (n9) circle [radius=.5mm];

\draw [->,dashed,thick,blue] (n4) edge[bend left] (n2);
\draw [->,dashed,thick,blue] (n5) arc [radius=3.5mm, start angle=450, end angle= 100]  (n5);
\draw [->,dashed,thick,blue] (n8) arc [radius=3.5mm, start angle=0, end angle= 340]  (n8);
\draw [->,dashed,thick,blue] (n2) edge[bend left] (n3);
\draw [->,dashed,thick,blue] (n3) edge[bend left] (n4);
\draw [->,dashed,thick,blue] (n1) arc [radius=3.5mm, start angle=450, end angle= 100]  (n1);

\draw [->,dashed,thick,blue] (n6) edge[bend right] (n7);

\draw [->,dashed,thick,blue] (n7) edge[bend right] (n8);

\draw [->,dashed,thick,blue] (n9) edge[bend left] (n7);

\end{scope}

\begin{scope}[xshift=5cm]
\node [
label=below right:{${(\bot,\top)}$}, label=below:{} ] (n1)  {} ;
\node [above right of=n1,label=right:{${(\bot,c)}$}, label=below:{ }] (n2)  {} ;
\node [above left of=n1,label=left:{${(\bot,a)}$}, label=below:{ }] (n3)  {} ;
\node [above right of=n3,label=left:{${(\bot,b)}$}, label=below:{ }] (n4)  {} ;
\node [above of=n4, label=left:{${(\bot,\bot)}$}, label=below:{ }] (n5)  {} ;
\node [above of=n5, label=left:{${(b,\bot)}$}, label=below:{ }] (n6)  {} ;
\node [above right of=n6,label=right:{${(c,\bot)}$}, label=below:{ }] (n7)  {} ;
\node [above left of=n6,label=left:{${(a,\bot)}$}, label=below:{ }] (n8)  {} ;
\node [above right of=n8, label=above right:{${(\top,\bot)}$}, label=below:{ }] (n9)  {} ;
\draw  (n1) -- (n2);
\draw  (n1) -- (n3);
\draw  (n2) -- (n4);
\draw  (n3) -- (n4);
\draw  (n4) -- (n5);
 \draw  (n5) -- (n6);
\draw  (n6) -- (n7);
\draw  (n6) -- (n8);
\draw  (n7) -- (n9);
\draw  (n8) -- (n9);

\draw [fill] (n1) circle [radius=.5mm];
\draw [fill] (n2) circle [radius=.5mm];
\draw [fill] (n3) circle [radius=.5mm];
\draw [fill] (n4) circle [radius=.5mm];
\draw [fill] (n5) circle [radius=.5mm];
 \draw [fill] (n6) circle [radius=.5mm];
  \draw [fill] (n7) circle [radius=.5mm];
   \draw [fill] (n8) circle [radius=.5mm];
    \draw [fill] (n9) circle [radius=.5mm];

\draw [->,dashed,thick,red] (n4) edge[bend left] (n2);
\draw [->,dashed,thick,red] (n5) arc [radius=3.5mm, start angle=450, end angle= 100]  (n5);
\draw [->,dashed,thick,red] (n3) arc [radius=3.5mm, start angle=0, end angle= 340]  (n3);
\draw [->,dashed,thick,red] (n2) edge[bend left] (n3);
\draw [->,dashed,thick,red] (n6) edge[bend right] (n7);

\draw [->,dashed,thick,red] (n8) edge[bend right] (n6);

\draw [->,dashed,thick,red] (n7) edge[bend right] (n8);

\draw [->,dashed,thick,red] (n9) arc [radius=3.0mm, start angle=270, end angle= 610]  (n2);
\draw [->,dashed,thick,red] (n1) edge[bend right] (n2);

\end{scope}
\end{tikzpicture}
\caption{The Hasse diagram of the modal Nelson algebra $\mathbf{N}(\mathbf{M})$. The behavior of the $\blacksquare$ operator is depicted in blue on the left, and the behavior of its dual $\blacklozenge$ operator is shown in red on the right.}
\label{fig2*}
\end{figure}
\end{example}

As in the non-modal case, when representing a modal Nelson lattice as a twist structure, we obtain an embedding that is onto if and only if the Nelson lattice reduct is centered. To achieve an isomorphism, it is necessary to consider a boolean filter of the Heyting algebra that is compatible with the modal operators, similar to the approach in \cite{JanRivi2014}.

\begin{lemma} \label{ModalHeyting-MN3}
Let $\mathbf{M}=\langle \He, \square, \Diamond \rangle$ be a modal Heyting algebra, and let $F$ be a boolean filter satisfying the following condition:
\begin{align}
    \text{If } a \wedge b = \bot \text{ and } a \vee b \in F, \text{ then } \square a \vee \Diamond b \in F. \tag{F}\label{F}
\end{align}
Then, $\mathbf{N}(\mathbf{M}, F) = \langle \R{\He}{F}, \blacksquare\rangle$ is a subalgebra of $\mathbf{N}(\mathbf{M}) = \langle \RH{\He}, \blacksquare \rangle$.
\end{lemma}
\begin{proof}
By Theorem \ref{TwistRepresentation}, it follows that the reduct $\langle \RN{\mathbf{H}}{F}, \wedge, \vee, *, \To, \bot, \top \rangle$ is a Nelson lattice, which is a subalgebra of $\RH{\He}$. 

Now, assume that $(a, b) \in \RN{\He}{F}$, then it follows that $a \wedge b = \bot$ and $a \vee b \in F$. Applying property \eqref{mH'} from Lemma \ref{equivAxrule1} and using Condition \eqref{F}, we obtain $\Box a \wedge \Diamond b = \bot$ and $\Box a \vee \Diamond b \in F$. Thus, $(\Box a, \Diamond b)\in \RN{\He}{F}$. 
\end{proof}

The previous results allow us to extend Sendlewski’s representation of Nelson lattices, introduced in Theorem \ref{TwistRepresentation}, by incorporating modal operators in a natural and straightforward manner. 

\begin{theorem}\label{isomorfismo}
Let $\mathbf{N} = \langle \A,\blacksquare\rangle$ be a modal Nelson lattice. Then, $\mathbf{N}$ is isomorphic to $\mathbf{N}(\mathbf{M}_\mathbf{N}^*, F^*)$, where $F^*$ is the filter defined by $F^*=\{(a \vee \sim a)^2 : a \in A \}$. 
\end{theorem}

\begin{proof}
Since $F^*$ contains the dense elements, it follows that $F^*$ is a boolean filter of $\mathbf{H}^*$. By Lemmas \ref{lemma Heytingfilter} and \ref{ModalHeyting-MN3}, we obtain that $\mathbf{N}(\mathbf{M}_\mathbf{N}^*, F^*)$ is an MN-lattice.

Define the mapping $h: A \to \RN{\He^*}{F^*}$ by 
\[
h(a) = (a^2, (\sim a)^2), \quad \text{for all } a \in A.
\]
We verify that $h$ is well defined: 
\begin{align*}
    a^2 \wedge^* (\sim a)^2 &= (a^2 \wedge (\sim a)^2)^2 = (a \wedge \sim a)^2 = \bot, \\
    a^2 \vee^* (\sim a)^2 &= (a^2 \vee (\sim a)^2)^2 = (a \vee \sim a)^2 \in F^*.
\end{align*}
By Theorem \ref{MappingNelsonHeyting}(1), we know that the Nelson lattice reduct $\A$ is isomorphic to $\R{\He^*}{F^*}$.

It remains to verify the preservation of the modal operator. For any $a \in A$, we have:
\begin{align*}
    h(\blacksquare a) &= ((\blacksquare a)^2, (\sim \blacksquare a)^2) \\
    &= ((\blacksquare a)^2, (\blacklozenge \sim a)^2) \\
    &= (\Box^* a^2, \Diamond^* (\sim a)^2) \\
    &= \blacksquare h(a).
\end{align*}

Thus, $h$ is an isomorphism of modal Nelson lattices.
\end{proof}

\begin{theorem} \label{Canonicalmapping}
    Let $\mathbf{M} = \langle \He, \square, \Diamond \rangle$ be a modal Heyting algebra, and let $F \subseteq H$ be a boolean filter satisfying Condition \eqref{F}. Consider the modal Nelson lattice $\mathbf{N}(\mathbf{M},F) = \langle \R{\He}{F},\blacksquare\rangle$. Then, $\mathbf{M}$ is isomorphic to $\mathbf{M}_{\mathbf{N}(\mathbf{M},F)}^* = \langle \He^*, \square^*, \Diamond^* \rangle$, where the latter is the Heyting algebra obtained from $\mathbf{N}(\mathbf{M},F)$ as described in Lemma \ref{lemma Heytingfilter}. The isomorphism $h\colon H \to H^*$ is given by  
    \[
    h(a) = (a, -a),
    \]
    for all $a \in H$. Moreover, $h[F] = F^*$.
\end{theorem}
\begin{proof} 
We know that the non-modal reducts are isomorphic under the map $h\colon H \to H^*$ defined by $h(a) = (a, -a)$ for all $a \in H$. Additionally, it is known that $h[F] = F^*$.

Now, for any $a \in H$, we have:
\[
h(\square a) = (\square a, -\square a) = (\blacksquare (a, -a))^2 = \square^*(a, -a) = \square^* h(a).
\]
The proof that $h(\Diamond a) = \Diamond^* h(a)$ follows analogously.
\end{proof}

These results leads us to state that the relationship between modal Nelson and modal Heyting algebras has a categorical nature. For the notions of category theory used in what follows we refer the reader to \cite{MacL98}.

Let denote by $\MN$ the category of modal Nelson residuated lattices and their homomorphisms and by $\mathbb{TW}$ the category whose objects are pairs $P=(\mathbf{M},F)$ of modal Heyting algebras $\mathbf{M}$ and boolean filters $F$ that satisfy Condition \eqref{F}. A morphism between two pairs $P_1=(\mathbf{M}_1,F_1)$ and $P_2=(\mathbf{M}_2,F_2)$ is a homormorphism of modal Heyting algebras $h\colon \mathbf{M_1}\to \mathbf{M_2}$ that satisfies $h[F_1]\subseteq F_2$. 

We can define the following functors:
\begin{itemize}
    \item $\mathbf{F}\colon \MN \to \mathbb{TW}$: For each $\mathbf{N}=\langle\A,\blacksquare\rangle \in \MN$, $\mathbf{F}(\mathbf{N}):=(\mathbf{M}_\mathbf{N}^*,F_\mathbf{N}^*)$, where $\mathbf{M}_\mathbf{N}^*$ is the modal Heyting algebra whose universe is $H_\A^* = \{  a^2 : a\in A \}$, and $F_\mathbf{N}^* = \{(a\vee \sim a)^2 : a \in A\}$.  \\
    For each morphism $g: \mathbf{N}_1 \to \mathbf{N}_2$ between modal Nelson algebras $\mathbf{N}_1=\langle\A_1,\blacksquare_1\rangle$ and $\mathbf{N}_2=\langle\A_2,\blacksquare_2\rangle$, we define the morphism $\mathbf{F}(g)\colon \mathbf{M}_{\mathbf{N}_1}^*\to \mathbf{M}_{\mathbf{N}_2}^* $ by the restriction $\mathbf{F}(g) := g\upharpoonright H_{\A_1}^*$. 
    \item $\mathbf{E}: \mathbb{TW} \to \MN$: For each $P=(\mathbf{M},F) \in \mathbb{TW}$, the modal Nelson lattice $\mathbf{E}(P):=\mathbf{N}(\mathbf{M},F)$. \\ 
    For each morphism $h:\mathbf{M}_1 \to \mathbf{M}_2$, we define $\mathbf{E}(h)\colon \mathbf{N}(\mathbf{M_1},F_1)\to \mathbf{N}(\mathbf{M_2},F_2)$ by:  
    \[
    \mathbf{E}(h)(x,y) := (h(x), h(y)).
    \] 
\end{itemize}

 Let $\mathbf{N}=\langle\A,\blacksquare\rangle$ be a modal Nelson lattice. By Theorem \ref{isomorfismo}, the map $\alpha_\mathbf{N}\colon \mathbf{N}\to \mathbf{N}(\mathbf{M_\mathbf{N}}^*,F_\mathbf{N}^*)$ defined by
    \[\alpha_\mathbf{N}(a)=(a^2,(\sim a)^2)\] is an isomorphism of modal Nelson lattices. On the other hand, let $P=(\mathbf{M},F)\in \mathbb{TW}$, the map $\beta_P:\mathbf{M}\to \mathbf{M}_{\mathbf{N}(\mathbf{M},F)}^*$ defined by
    \[\beta_P(a)=(a,-a)\]
is an isomorphism of modal Heyting algebra that satisfies $\beta_P[F]=F_{\mathbf{N}(\mathbf{M},F)}^*$.
Moreover, diagrams in Figure \ref{fig:comm-diagrams} commute.
\begin{figure}[h]
    \centering
    \begin{tikzcd}
        \mathbf{N}_1 \arrow[r,"\alpha_{\mathbf{N}_1}"] \arrow[d,"g"] & \mathbf{N}(\mathbf{M}_{\mathbf{N}_1}^*, F_{\mathbf{N}_1}^*) \arrow[d,"\mathbf{E}\circ\mathbf{F}(g)"] \\ 
        \mathbf{N}_2 \arrow[r,"\alpha_{\mathbf{N}_2}"] & \mathbf{N}(\mathbf{M}_{\mathbf{N}_2}^*,F_{\mathbf{N}_2}^*)
    \end{tikzcd}
    \hspace{2cm}
    \begin{tikzcd}
        \mathbf{M_1} \arrow[r,"\beta_{P_1}"] \arrow[d,"h"] & \mathbf{M}_{\mathbf{N}(\mathbf{M_1},F_1)}^* \arrow[d,"\mathbf{F}\circ\mathbf{E}(h)"] \\ 
        \mathbf{M_2} \arrow[r,"\beta_{P_2}"] & \mathbf{M}_{\mathbf{N}(\mathbf{M_2},\mathbf{F_2})}^*
    \end{tikzcd}
    \caption{Commutative diagrams for $\mathbf{E}$ and $\mathbf{F}$ where $\mathbf{N}_1=\langle\A_1,\blacksquare_1\rangle$, $\mathbf{N}_2=\langle\A_2,\blacksquare_2\rangle$, and $P_1=( \mathbf{M}_1,F_1)$ and $P_2=( \mathbf{M}_2,F_2)$.}
    \label{fig:comm-diagrams}
\end{figure}

Thus, we conclude:
\begin{theorem}
   The categories $\MN$ and $\mathbb{TW}$ are equivalent. 
\end{theorem}

\begin{proof}
    Let $\mathbf{N} $ be a modal Nelson lattice. From Theorem \ref{isomorfismo}, it follows that
the map $\alpha_\mathbf{N} \colon \mathbf{N} \to \mathbf{N}(\mathbf{M}_\mathbf{N}^*,F_\mathbf{N}^*)$ defines a natural isomorphism between $\mathbf{E} \circ \mathbf{F}$ and
$\Id_\MN$ . On the other hand, let $P=(\mathbf{M},F)$ be a modal Heyting algebra with a boolean filter. Then, from Theorem \ref{Canonicalmapping}, we get that the map $\beta_P \colon \mathbf{M} \to \mathbf{M}_{\mathbf{N}(\mathbf{M},F)}^*$ leads
to a natural isomorphism between $\mathbf{F} \circ \mathbf{E}$ and $\Id_{\mathbb{TW}}$. This concludes the proof.
\end{proof}

\section{Two interesting subvarieties of \texorpdfstring{$\mathcal{MN}$}{MN}}\label{sec5}

This section will examine two specific subvarieties of $\mathcal{MN}$ whose non-modal reducts correspond to important subvarieties of Nelson lattices.

Following the literature, we define the following three unary term functions for each Nelson lattice $\A$:
\begin{align*}
    \nabla(x) &= \sim(\sim x^2)^2, \\
    \Delta(x) &= (\sim(\sim x)^2)^2, \\
    \phi(x) &= \Delta(x) \wedge (\nabla(x \vee \sim x) \vee x).
\end{align*}
Now, let $\He$ be a Heyting algebra and consider the corresponding Nelson lattice $\A = \RH{\He}$. Computing the operators defined above on $\A$, we obtain:
\begin{align}
    \nabla(x,y) &= (- - x, - x),\\
    \Delta(x,y) &= (- y, - - y), \\
    \phi(x,y) &= (- - x, - - y). \label{fipares}
\end{align}
Here, as before, the negation operator $-$ is given by $- x = x \rightharpoonup \bot$.

We recall the following theorem, which endows the image of $\phi$ with Nelson lattice operations.

\begin{theorem} \cite[Theorem 5.4]{Busaniche} \label{PhiHomo}
Let $\A = \langle A,\vee,\wedge,*,\To,\top,\bot\rangle$ be a Nelson lattice. Then $\Phi(\A)= \langle \phi(A),\vee',\wedge',*,\To,\top,\bot\rangle$ is Nelson lattice where
$$\phi(A)= \{y \in A \  | \ y = \phi(x) \mbox{ for some } x \in A \} $$
and for each $\star \in \{\vee,\wedge \}$ the operation $\star'$ is given by $x \star' y = \phi(x \star y)$. In addition, it also results that $\phi$ is a homomorphism from $\A$ onto $\Phi(\A)$.
\end{theorem}

\subsection{Modal \texorpdfstring{$\phi$}{phi}-normal Nelson lattices}

First, we consider the variety of $\phi$-normal Nelson lattices, introduced in \cite{Gor85}\footnote{In fact, the original authors use the term \emph{normal Nelson algebras}. We instead introduce the prefix $\Phi$ to avoid confusion with another use of the term \emph{normal} that occurs earlier in our paper.} in the context of Nelson algebras, characterized by the equation:
\begin{equation}\label{NormalN3}
    \nabla x  = \Delta x.
\end{equation}
which is equivalent to:
\begin{equation*}
    \sim x^2 \To x^2 = (\sim x \To x)^2.
\end{equation*}

By Corollary 5.9 in \cite{Busaniche}, we know that a Nelson lattice $\A$ satisfies \ref{NormalN3} if and only if $\Phi(\A)$ is a Boolean algebra. In \cite{Busaniche}, the authors proved the following lemma and theorem, which will be useful in this section.

\begin{lemma} (\cite[Lemma 6.2]{Busaniche})  \label{DensequivHeyting}
Let $\He$ be a Heyting algebra. The following are equivalent conditions for all $x,y\in H$:
\begin{enumerate}
    \item $x\wedge y=\bot$ and $x\vee y\in D(\He)$, 
    \item $x\wedge y=\bot$ and $- x\wedge - y=\bot$,
    \item $- x= - - y$,
    \item $- x \rightharpoonup x = - y$.
\end{enumerate}
\end{lemma}

\begin{theorem} (\cite[Theorem 6.3]{Busaniche}) \label{CharacNormalNelsonAlgebras}
A Nelson lattice $\A$ satisfies $\nabla x=\Delta x$ if and only if there exists a Heyting algebra $\He$ such that $\A$ is isomorphic to $\R{\He}{D(\He)}$. 
\end{theorem}

A \emph{modal $\phi$-normal Nelson lattice} is a modal Nelson lattice $\mathbf{N}=\langle\A,\blacksquare\rangle$ that satisfies equation \eqref{NormalN3}. 

Now, consider a modal Heyting algebra $\mathbf{M}=\langle \He,\square,\Diamond\rangle$ such that, for all $a\in H$, the following conditions hold:
\begin{align}
    - - \square a &= - \Diamond - a,\label{N1}\\
    - \square - a &= - - \Diamond a.\label{N2}
\end{align}

\begin{lemma} \label{normal1}
Let $\mathbf{M}=\langle \He, \square,\Diamond\rangle$ be a modal Heyting algebra satisfying equations \eqref{N1} and \eqref{N2}. Then $\mathbf{N}(\mathbf{M},D(\He))$ is a subalgebra of $\mathbf{N}(\mathbf{M})$.
\end{lemma}

\begin{proof}
Let $a,b\in H$ such that $a\wedge b=\bot$ and $a\vee b\in D(\He)$. We will prove that $\square a\vee \Diamond b\in D(\He)$. By Lemma \ref{DensequivHeyting}, we get that $- a=-- b$. By equation \eqref{N1}, we have that $-\square a=--\Diamond -a$. Replacing $-a$, we get $-\square a=--\Diamond --b$. By equation \eqref{N2}, we have $--\Diamond --b=-\square -b=--\Diamond b$. Thus, $-\square a=--\Diamond b$ and by Lemma \ref{DensequivHeyting}, we get $\square a\vee \Diamond b\in D(\He)$. Therefore, by Theorem \ref{ModalHeyting-MN3}, $\mathbf{N}(\mathbf{M},D(\He))$ is a modal Nelson lattice.
\end{proof}

\begin{lemma}\label{normal2}
Let $\mathbf{N}=\langle \A,\blacksquare\rangle$ be a modal Nelson lattice, and let
\[
\mathbf{M}^*=\langle \mathbf{H}^*,\square^*,\Diamond^*\rangle
\]
be the associated modal Heyting algebra. Then $\mathbf{M}^*$ satisfies Equations \eqref{N1} and \eqref{N2} if and only if the filter $D(\He^*)$ satisfies Condition~\eqref{F}.
\end{lemma}

\begin{proof}
Suppose that $D(\He^*)$ satisfies Condition~\eqref{F}. We will prove that $\mathbf{M}^*=\langle \He^*,\square^*,\Diamond^*\rangle$ satisfies equations \eqref{N1} and \eqref{N2}. Let $a\in H^*$. On the one hand, we have $-^*-^* \square^* a=-^*(\sim (\blacksquare a)^2 )^2=(\sim (\sim (\blacksquare a)^2 )^2)^2$. On the other hand $-^*\Diamond^*-^* a=-^*(\blacklozenge (\sim a)^2)^2=-^*(\blacklozenge \sim a)^2=(\sim (\blacklozenge \sim a)^2)^2=(\sim (\sim \blacksquare a)^2)^2$. By equation \eqref{NormalN3}, $-^*-^* \square^* a=(\sim (\sim (\blacksquare a)^2 )^2)^2=(\sim (\sim \blacksquare a)^2)^2=-^*\Diamond^*-^* a$. Analogously, we can prove $-^* \square^*-^* a=-^*-^*\Diamond^* a$.

The other direction follows immediately from Lemma \ref{normal1}.
\end{proof}

Finally, we obtain the following theorem, extending Theorem \ref{CharacNormalNelsonAlgebras} to the modal language.

\begin{theorem}
A modal Nelson lattice $\mathbf{N}=\langle\A,\blacksquare\rangle$ satisfies the equation $\Delta x = \nabla x$ if and only if there exists a modal Heyting algebra $\mathbf{M}=\langle \He, \square,\Diamond\rangle$ satisfying equations \eqref{N1} and \eqref{N2}, such that $\mathbf{N}$ is isomorphic to $\mathbf{N}(\mathbf{M},D(\He))$.
\end{theorem}

\begin{proof}
It follows immediately from Lemmas \ref{normal1} and \ref{normal2} and Theorems \ref{isomorfismo} and \ref{CharacNormalNelsonAlgebras}.
\end{proof}

\subsection{Modal \texorpdfstring{$\phi$}{phi}-regular Nelson lattices}

Now, we are going to extend the notion of $\phi$-regular algebra to the modal context. A \emph{$\phi$-regular Nelson lattice} is a Nelson lattice $\A$ such that $\Phi(\A)$ is a subalgebra of $\A$, i.e.
$\phi(x \vee y)=\phi(x)\vee \phi(y)$ and $\phi(x\wedge y)=\phi(x)\wedge \phi(y)$ for all $x,y \in A$. The variety of $\phi$-regular Nelson lattices was studied in \cite{Busaniche}, which is characterized by:
\begin{equation}
    (\sim x^2)^2 \vee (\sim (\sim x^2)^2)^2 = \top
\end{equation}

We are going to say that $\mathbf{N}=\langle\A,\blacksquare\rangle$ is a \emph{modal $\phi$-regular Nelson lattice} if the non-modal reduct of $\A$ is a $\phi$-regular Nelson lattice and for any $x \in A$:
$$\blacksquare \phi(x) = \phi (\blacksquare x),$$ 
i.e., $\Phi(\A)$ with the restricted modal operator is modal Nelson subalgebra of $\mathbf{N}$.

\begin{definition} \label{crisp-wit}
A modal Heyting algebra $\mathbf{M}=\langle \He,\square,\Diamond\rangle$ is said to be \textit{crisp-witnessed} if it satisfies the equations
\begin{equation} \label{modaldoblenegation}
 --\Box x = \Box --x \hspace{2mm} \mbox{ and } \hspace{2mm} --\Diamond x = \Diamond --x   \\
\end{equation}
for every $x \in H$.
\end{definition}

The equations $--\Box x \rightharpoonup \Box --x = \top$ and $\Diamond --x \rightharpoonup  --\Diamond x = \top$  have been studied in \cite{CaiRod2010} as a way of axiomatizing the box-fragment and diamond-fragment of minimal modal G\"odel logic, respectively. It is well known that the reciprocal of both equations (so-called double negation shift) are not theorems of intuitionistic predicate logic but they characterize some interesting fragments. In particular, it is shown in \cite{CaiRod2010} that $\Box --x \rightharpoonup --\Box x = \top$ is strongly complete for 0-witnessed models and it is conjectured that $--\Diamond x \rightharpoonup  \Diamond --x= \top$ characterizes crisp frames.

\begin{theorem} \label{CharRegNelsonlattice}
A modal Nelson lattice $\mathbf{N}=\langle \A,\blacksquare\rangle $ is a modal $\phi$-regular Nelson lattice if and only if the associated modal Heyting algebra $\mathbf{M}_\mathbf{N}=\langle \He_\A,\square,\Diamond\rangle$ is crisp-witnessed and the Heyting algebra $\He_\A$ satisfies the Stone identity $-x \vee --x = \top$.
\end{theorem}

\begin{proof}
By \cite[Theorem 5.12]{Busaniche}, we only need to prove the equalities for the modal operators. Let $F$ be a boolean filter of $\He_\A$ such that $\mathbf{N}$ is isomorphic to $\mathbf{N}(\mathbf{M}_\mathbf{N},F)$. 

We can see that, by (\ref{fipares}) for any pair  $(x,y)\in R(\He_\A,F)$, $\phi(\blacksquare (x,y)) = \phi(\Box x,\Diamond y) = (- - \Box x, - -\Diamond y)$  and 
$\blacksquare \phi(x,y) = (\Box- - x, \Diamond - - y)$. Therefore, if the modal Heyting algebra $\langle \He_\A,\square,\Diamond\rangle$ satisfies equations in (\ref{modaldoblenegation}) the result follows.\\
To prove the converse, let $x\in H_\A$. Since $(x,-x)\in R(\He_\A,F)$, we have $\phi(\blacksquare (x,-x)) = \blacksquare \phi(x,-x)$. Then by equation (\ref{fipares}), we obtain $(- - \Box x, - -\Diamond -x)= (\Box- - x, \Diamond - x)$ which implies that the equation $--\square x=\square --x$ holds in $\mathbf{M}_\mathbf{N}$. Analogously, since $(-x,x)\in R(\He_\A,F)$, $\phi(\blacksquare (-x,x)) = \blacksquare \phi(-x,x)$. It follows that $(- - \Box -x, - -\Diamond x)= (\Box-  x, \Diamond -- x)$ and therefore the equation $-- \Diamond x=\Diamond --x$ holds in $\mathbf{M}_\mathbf{N}$.
\end{proof}

\section{Topological dualities}\label{sec6}

Recall that a \emph{Priestley space} is a compact ordered space $\mathcal{X} = \langle X, \tau, \leq \rangle$ that satisfies the separation condition, i.e., for every $x,y \in X$ such that $x \not\leq y$, there exists a clopen up-set $U \subseteq X$ with $x \in U$ and $y \not\in U$. 
We will denote the up-set and down-set generated by a set $U\subseteq X$ by $\ua{U} =\{x\in X:(\exists u\in U) \ x\leq x\}$, and $\da{U}=\{x\in X:(\exists u\in U) \ x\leq u\} $, respectively.
Moreover, we will denote by $\mathcal{U}(X)$ the family of clopen up-sets and by $\mathcal{D}(X)$ the family clopen down-sets. 
An \emph{Esakia space} is a Priestley space satisfying the additional condition that for every clopen $U \in X$, the set $\da{U}\in \mathcal{D}(X)$. The set of clopen up-sets of an Esakia space forms a distributive lattice, which becomes a Heyting algebra when endowed with the following implication operation: for clopen up-sets $U, V \subseteq X$, we define $U \to V := \{x \in X : \ua{x} \cap U \subseteq V\}$.
Let $\mathcal{X}$ and $\mathcal{Y}$ be Esakia spaces. A map $f \colon X\to Y$
is an \emph{Esakia function} if it is continuous, order-preserving and satisﬁes that $\uaY{f(x)}\subseteq f[\uaX{x}]$ for every
$x\in X$. 

Conversely, every Heyting algebra $\He$ gives rise to an Esakia space $\langle X(\He), \tau_\He, \subseteq \rangle$, where $X(\He)$ is the poset of prime filters of $\He$, ordered by inclusion, and $\tau_\He$ is is the topology generated by the sub-basis: $$\{\sigma_\He(a) : a \in H\} \cup \{\X{\He}\setminus \sigma_\He(a) : a \in H\}$$ with $\sigma_\He(a)= \{ P \in \X{\He} : a \in P \}$. The map $\sigma_\He: \He \to \mathcal{U}(\X{\He})$ defined as above is an isomorphism of Heyting algebras.

The categories of Heyting algebras with homomorphisms and Esakia spaces with Esakia functios are dually equivalent. If $h\colon \He_1\to \He_2$ is a homomorphism of Heyting algebras, then the map $X(h) \colon \X{\He_2}\to\X{\He_2}$
between the corresponding Esakia spaces deﬁned by $X(h)(P)=h^{-1}[P]$ for every $P\in \X{\He_2}$ is an Esakia function. Conversely, if $f \colon X_1\to X_2$ is an Esakia function, then the map $h(f) \colon \mathcal{U}(X_2)\to \mathcal{U}(X_1)$ deﬁned by $h(f)(U)=f^{-1}[U]$ for every clopen up-set $U$ of $X_2$ is a Heyting algebra homomorphism.

\subsection{Topological duality for modal Heyting algebras}
In this section, we will extend the topological duality to modal Heyting algebras. To achieve this, we will make use of neighbourhood functions, a common tool for interpreting non-normal modal operators. We will follow the approach outlined in \cite{Oli2020}, where a semantic characterization of the logic $\textbf{IE}_3$ is provided in terms of neighbourhood models that include two distinct neighbourhood functions, each corresponding to one of the two modalities.

Let $\mathbf{M}= \langle \He,\square , \Diamond\rangle$ be a modal Heyting algebra. 
For each operation $\bullet \in \{\square, \Diamond\}$ we define a neighbourhood function $\eta_\bullet : \X{\He}\mapsto \mathcal{P}(\mathcal{P}(\X{\He}))$, i.e. we associate to each prime filter a set of prime filters which is called its neighbourhood in the following way: 
\begin{align} \label{defneighb}
\eta_\square(P) &:=\{\sigma_\He(a) : \square a \in P \},\\ \label{defneighd}
\eta_\Diamond(P) &:=\mathcal{D}(\X{\He})\setminus\{ \X{\He}\setminus \sigma_\He(a) : \Diamond a \in P \}.
\end{align} 

Using the neighbourhood function $\eta_\bullet$, we can define an algebraic operator $\bullet \in \{\Box, \Diamond\}$ in the Heyting algebra of clopen up-sets of $\X{\He}$, denoted by $\mathcal{U}(\X{\He})$. For any $U \in \mathcal{U}(\X{\He})$, the operators are defined as follows:

\begin{align} 
\square_{\eta_\square}(U) &:= \{ P \in \X{\He} : U \in \eta_\square(P)\}, \label{topmodalmapb}\\
\Diamond_{\eta_\Diamond}(U) &:= \{ P \in \X{\He} : \X{\He}\setminus U \notin \eta_\Diamond(P)\}. \label{topmodalmapd}
\end{align}

The following proposition shows that, using the above definitions, we can extend the isomorphism $\sigma_\He$ for Heyting algebras to include modal operators as well:

\begin{proposition}
Let $\mathbf{M}= \langle \He,\square , \Diamond\rangle$ be a modal Heyting algebra. Then, for every $a \in H$: 
\begin{enumerate}
    \item $\sigma_\He({\square a}) = \square_{\eta_\square}(\sigma_\He(a))$.
    \item $\sigma_\He({\Diamond a}) = \Diamond_{\eta_\Diamond}(\sigma_\He(a))$.
\end{enumerate}
\end{proposition}
\begin{proof}

1. Let $P\in \X{\He}$. By \eqref{defneighb}, $\square a \in P$ if and only if $\sigma_\He(a) \in \eta_\square(P)$. It follows from \eqref{topmodalmapb}, $\sigma_\He(a) \in \eta_\square(P)$ if and only if $P \in \square_{\eta_\square}(\sigma_\He(a))$. Therefore $\sigma_\He({\square a}) = \square_{\eta_\square}(\sigma_\He(a))$.

2. Let $P\in \X{\He}$.  By \eqref{defneighd}, $\Diamond a \in P$ if and only if $\X{\He}\setminus \sigma_\He(a) \notin \eta_\Diamond(P)$. Then, it follows from \eqref{topmodalmapd}, $\X{\He}\setminus \sigma_\He(a) \notin \eta_\Diamond(P)$ if and only if $P\in \Diamond_{\eta_\Diamond}(\sigma_\He(a))$. Therefore, $\sigma_\He({\Diamond a}) = \Diamond_{\eta_\Diamond}(\sigma_\He(a))$.
\end{proof}


\begin{proposition}\label{propeta}
    Let $\mathbf{M}=\langle\He,\square,\Diamond\rangle$ be a modal Heyting algebra. Then, for any $P\in\X{\He}$ and $U,V\in\mathcal{U}(\X{\He})$, the following condition holds:
    \[\text{if }U\in \eta_\square(P)\text{, then }\da{U}\cup (\X{\He}\setminus V)\in \eta_\Diamond(P).\]
\end{proposition}

\begin{proof}
    Let $U\in \eta_\square(P)$ and let $V\in \mathcal{U}(\X{\He)}$. By definition, there exists $a\in H$ such that $\sigma_\He(a)=U $ and $\square a\in P$. Suppose $\da{U}\cup (\X{\He}\setminus V)\notin \eta_\Diamond(P)$. By \eqref{defneighd}, there exists $b\in H$ such that $\da{U}\cup (\X{\He}\setminus V)=\X{\He}\setminus\sigma_\He(b)$ and $\Diamond b\in P$. Then, $(\X{\He}\setminus \da{U})\cap V=\sigma_\He(b)$. It is easy to see that $\X{\He}\setminus \da{U}=\X{\He}\setminus \da{\sigma_\He(a)=\sigma_\He(-a)}$ and thus $\sigma_\He(b)\subseteq\sigma_\He(-a)$. From \eqref{mH'}, since $a\wedge b=0$ it follows that $\square a\wedge\Diamond b=\bot\notin P$, which contradicts the fact that $\square a,\Diamond b\in P$.
\end{proof}


The condition in Proposition \ref{propeta} can be viewed as a topological counterpart of the Strong Interaction Property for neighbourhood semantics introduced in \cite{Oli2020}. 
Analyzing these properties, we introduce the topological spaces that we will later prove to be dual to modal Heyting algebras.

\begin{definition}\label{ME-space}
    A \emph{modal Esakia space} (ME-space) is a structure $\mathcal{X}=\langle X,\tau,\leq,\eta_1,\eta_2\rangle$ such that $\langle X,\leq,\tau\rangle$ is an Esakia space and $\eta_1,\eta_2\colon X\to \mathcal{P}(\mathcal{P}(X))$ are neighbourhood functions satisfying the following properties for all $x\in X$ and for all $U,V\in \mathcal{U}(X)$:
    \begin{enumerate}
        \item $\eta_1(x)\in \mathcal{P}(\mathcal{U}(X))$ and $\eta_2(x)\in \mathcal{P}(\mathcal{D}(X))$,
        \item $\square_{\eta_1}(U)= \{ x \in X : U \in \eta_1(x)\}$, and $\Diamond_{\eta_2}(U)=\{ x \in X : X\setminus U \notin \eta_2(x)\}$ are clopen up-sets, 
        \item  if $U\in \eta_1(x)$, then $\da{U}\cup (X\setminus V)\in \eta_2(x)$.
    \end{enumerate}
\end{definition}

As a consequence of the previous theorems, we obtain the following result:

\begin{theorem}
    For each modal Heyting algebra $\mathbf{M}=\langle \He,\square,\Diamond\rangle$, the structure $\mathcal{X}(\mathbf{M)}=\langle \X{\He},\tau_\He,\subseteq,\eta_\square,\eta_\Diamond\rangle$ is a modal Esakia space.
\end{theorem}

Now, we will prove that given a modal Esakia space, the algebra of clopen up-sets, endowed with the operators defined from the neighbourhood functions, is a modal Heyting algebra.

\begin{proposition}
    Let $\mathcal{X}=\langle X,\tau,\leq,\eta_1,\eta_2\rangle$ be a modal Esakia space. Then, the algebra $\langle \mathcal{U}(X),\square_{\eta_1},\Diamond_{\eta_2}\rangle$, where $\square_{\eta_1},\Diamond_{\eta_2}\colon\mathcal{U}(X)\to \mathcal{U}(X)$ are the operators defined by
    \begin{align}\label{squareeta}
        \square_{\eta_1}(U)&= \{ x \in X : U \in \eta_1(x)\},\\ \Diamond_{\eta_2}(U)&=\{ x \in X : X\setminus U \notin \eta_2(x)\},\label{diamondeta}
    \end{align}
    is a modal Heyting algebra.
\end{proposition}
\begin{proof}
    We only need to prove that \eqref{mH'} is satisfied. Let $U,V\in \mathcal{U}(X)$ such that $U\cap V=\emptyset$. Suppose that there exists $x\in X$ such that $x\in \square_{\eta_1}(U)\cap\Diamond_{\eta_2}(V)$. Then, $U\in \eta_1(x)$ and $X\setminus V\notin \eta_2(x)$. By Definition \ref{ME-space}, since $U\subseteq X\setminus V$ and $U\in \eta_1(x)$, we obtain $X\setminus V=\da{U}\cup (X \setminus V)\in \eta_2(x)$, which is a contradiction. Therefore $\square_{\eta_1}(U)\cap\Diamond_{\eta_2}(V)=\emptyset$.
\end{proof}

\begin{theorem}\label{natis}
    For each modal Heyting algebra $\mathbf{M}=\langle \He,\square,\Diamond\rangle$, the map $\sigma_\He \colon \He \to \mathcal{U}(\X{\He})$ defined by
    \[\sigma_\He(a)=\{P\in \X{\He}:a\in P\}\]
    is a modal
Heyting algebra isomorphism from $\langle \He,\square,\Diamond\rangle$ onto $\langle\mathcal{U}(\X{\He}),\square_{\eta_\square},\Diamond_{\eta_\Diamond}\rangle$.
\end{theorem}

\begin{theorem}\label{natep}
    Let $\mathcal{X}=\langle X,\tau,\leq,\eta_1,\eta_2\rangle$ be an ME-space and let $\epsilon_\mathcal{X}\colon X\to\X{\mathcal{U}(X)}$ be the map defined by:
    \[\epsilon_X(x)=\{U\in\mathcal{U}(X):x\in U\}.\]
    Then $\epsilon_X$ is an Esakia-homeomorphism between $\langle X,\tau,\leq\rangle$ and $\langle \X{\mathcal{U}(X)},\tau_{\mathcal{U}(X)},\subseteq\rangle$ that satisfies:
    \begin{enumerate}
        \item $\eta_{\square_{\eta_1}}(\epsilon_X(x))=\{\epsilon_X[U]:U\in\eta_1(x)\}$,
        \item $\eta_{\Diamond_{\eta_2}}(\epsilon_X(x))=\{\epsilon_X[D]:D\in\eta_2(x)\}$.
    \end{enumerate}
\end{theorem}
\begin{proof} It is known that $\epsilon_X$ is a homeomorphism and an order preserving isomorphism.\\
    1. Let $Z\in \eta_{\square_{\eta_1}}(\epsilon_X(x))$. By \eqref{defneighb}, $$Z=\sigma_{\mathcal{U}(X)}(U)=\{\epsilon_X(x): U\in \epsilon_X(x)
    \}=\{\epsilon_X(x): x\in U
    \}=\epsilon_X[U]$$ for some $U\in\mathcal{U}(X)$ such that $\square_{\eta_1}(U)\in \epsilon_X(x)$, i.e., $x\in \square_{\eta_1}(U)$. By \eqref{squareeta}, $U\in\eta_1(x)$. Therefore, $Z\in \{\epsilon_X[U]:U\in\eta_1(x)\}$. The other direction follows immediately from the definitions.\\
    2. Let $Z\in \eta_{\Diamond_{\eta_2}}(\epsilon_X(x))$. By \eqref{defneighd}, $$Z=\X{\mathcal{U}(X)}\setminus\sigma_{\mathcal{U}(X)}(U)=\{\epsilon_X(x): U\notin \epsilon_X(x)
    \}=\{\epsilon_X(x): x\notin U
    \}=\epsilon_X[X\setminus U]$$ for some $U\in\mathcal{U}(X)$ such that $\Diamond_{\eta_2}(U)\notin \epsilon_X(x)$, i.e., $x\notin \Diamond_{\eta_2}(U)$. By \eqref{diamondeta}, $X\setminus U\in\eta_2(x)$. Therefore, $Z\in \{\epsilon_X[D]:D\in\eta_2(x)\}$. The other direction follows immediately from the definitions.
\end{proof}

Now, we will introduce a category whose objects are modal Esakia spaces. Note that for a clopen up-closed set $V$ of $\X{\mathcal{U}(X)}$, we have the following equivalences:
\[
V \in \eta_{\square_{\eta_1}}(\epsilon_X(x)) \quad \text{if and only if} \quad \epsilon_X^{-1}[V] \in \eta_1(x),
\]
and
\[
(\X{\mathcal{U}(X)}\setminus V)\in \eta_{\Diamond_{\eta_2}}(\epsilon_X(x)) \quad \text{if and only if} \quad (X \setminus \epsilon_X^{-1}[V]) \in \eta_2(x).
\]

\begin{definition}\label{MEmorphism}
A map $f \colon X_1\to X_2$	 between two ME-spaces $\mathcal{X}_1=\langle X_1,\tau_1,\leq_1,\eta_1,\eta_2\rangle$ and $\mathcal{X}_2=\langle X_2,\tau_2,\leq_2,\eta'_1,\eta'_2\rangle$ is an ME-morphism iff $f$ is an Esakia function that additionally satisﬁes, for every $x\in X_1$ and every clopen up-set $U\in\mathcal{U}(X_2)$,
\begin{enumerate}
    \item $U\in \eta'_1	(f(x))$ if and only if $f^{-1}[U]\in\eta_1(x)$,
\item $(X_2\setminus U)\in \eta'_2(f(x))$ if and only if $(X_1\setminus f^{-1}[U])\in \eta_2(x)$.
\end{enumerate}
\end{definition}

It is straightforward to verify that the composition of ME-morphisms is again an ME-morphism and that the identity map of an ME-space is an ME-morphism. 

We now define the category $\mathbb{ME}$, whose objects are ME-spaces, and the morphisms are ME-morphisms.

\begin{proposition}
   Let $\mathbf{M}_1=\langle \He_1,\square_1,\Diamond_1\rangle$ and $\mathbf{M}_2=\langle \He_2,\square_2,\Diamond_2\rangle$ be two modal Heyting algebras and let $h\colon \He_1\to \He_2 $ be a homomorphism of modal Heyting algebras. Then $X(h)\colon \X{\He_2}\to\X{\He_1}$ is a ME-morphism.
\end{proposition}
\begin{proof}
    It is known that $X(h)$ is an Esakia function. We only need to prove 1 and 2 from Definition \ref{MEmorphism}. Let $a\in H_1$ and $P\in \X{\He_2}$.
    
    1. $\sigma_{\He_1}(a)\in \eta_{\square_1}(X(h)(P))$ if and only if $\square_1a\in h^{-1}(P)$. It follows that $h(\square_1a)=\square_2h(a)\in P$ and by \eqref{defneighb}, it is equivalent to $\sigma_{\He_2}(h(a))\in \eta_{\square_2}(P)$. In addition, $\sigma_{\He_2}(h(a))=(X(h))^{-1}[\sigma_{\He_1}(a)]$.\\
    2. $\X{\He_1}\setminus\sigma_{\He_1}(a)\in \eta_{\Diamond_1}(X(h)(P))$ if and only if $\Diamond_2a\notin h^{-1}(P)$. It follows that $h(\Diamond_1a)=\Diamond_2h(a)\notin P$ and by \eqref{defneighd}, it is equivalent to $\X{\He_2}\setminus\sigma_{\He_2}(h(a))\in \eta_{\Diamond_2}(P)$. In addition, $\X{\He_2}\setminus\sigma_{\He_2}(h(a))=\X{\He_2}\setminus(X(h))^{-1}[\sigma_{\He_1}(a)]$.
\end{proof}

\begin{proposition}
    Let $f \colon X_1\to X_2$ be an ME-morphism	between two ME-spaces $\mathcal{X}_1=\langle X_1,\tau_1,\leq_1,\eta_1,\eta_2\rangle$ and $\mathcal{X}_2=\langle X_2,\tau_2,\leq_2,\eta'_1,\eta'_2\rangle$. Then, the map $h(f) \colon \mathcal{U}(X_2)\to \mathcal{U}(X_1)$ deﬁned by $h(f)(U)=f^{-1}[U]$ is a modal Heyting algebra homomorphism. 
\end{proposition}
\begin{proof}
It is known that $h(f)$ is a Heyting algebra homomorphism. Let $U\in \mathcal{U}(X_2)$. Then,
    \begin{align*}
    \square_{\eta_1}(h(f)(U))&=\square_{\eta_1}(f^{-1}[U])\\&=\{x\in X_1:f^{-1}[U]\in\eta_1(x)\}\\&=\{x\in X_1:U\in\eta'_1(f(x))\}\\&=\{x\in X_1:f(x)\in\square_{\eta'_1}(U)\}\\&=f^{-1}[\square_{\eta'_1}(U)]\\&=h(f)(\square_{\eta'_1}(U)).
    \end{align*}
On the other hand,
    \begin{align*}
        \Diamond_{\eta_2}(h(f)(U))&=\Diamond_{\eta_2}(f^{-1}[U])\\&=\{x\in X_1:X_1\setminus f^{-1}[U]\in\eta_2(x)\}\\&=\{x\in X_1:X_2\setminus U\in\eta'_2(f(x))\}\\&=\{x\in X_1:f(x)\in\Diamond_{\eta'_2}(U)\}\\&=f^{-1}[\Diamond_{\eta'_2}(U)]\\&=h(f)(\Diamond_{\eta'_2}(U)).
    \end{align*}
    Since $U$ is arbitrary, it follows $h(f)$ is a modal Heyting algebra homomorphism.
\end{proof}

Let $\MH$ be the category of modal Heyting algebras whose morphsims are modal Heyting algebras homomorphisms.
Let $\mathbf{G}\colon \mathbb{ME}\to \MH$ be the contravariant functor that for each ME-space $\mathcal{X}=\langle X,\tau,\leq,\eta_1,\eta_2\rangle$, $\mathbf{G}(\mathcal{X}):=\langle \mathcal{U}(X),\square_{\eta_1},\Diamond_{\eta_2}\rangle$ and that for each ME-morphism $f\colon X_1\to X_2$, $\mathbf{G}(f)\colon \mathcal{U}(X_2)\to \mathcal{U}(X_1)$ is the modal Heyting algebra homomorphism $\mathbf{G}(f):=h(f)$.

On the other hand, Let $\mathbf{J}\colon \MH\to \mathbb{ME}$ be the contravariant functor that for each modal Heyting algebra $\mathbf{M}=\langle \He,\square,\Diamond\rangle$, $\mathbf{J}(\mathbf{M}):=\langle \X{\He},\tau_\He,\subseteq,\eta_\square,\eta_\Diamond\rangle$ and that for each modal Heyting algebra homomorphism $h\colon \He_1\to \He_2$, $\mathbf{J}(h)\colon \X{\He_2}\to \X{\He_1}$ is the ME-morphism $\mathbf{J}(h):=X(h)$. In this context, we obtain that diagrams in Figure \ref{fig3} commute.
    \begin{figure}[h]
    \centering
    \begin{tikzcd}
        H_1 \arrow[r,"\sigma_{\He_1}"] \arrow[d,"h"] & \mathcal{U}(\X{\He_1}) \arrow[d,"\mathbf{G}\circ\mathbf{J}(h)"] \\ 
        H_2 \arrow[r,"\sigma_{\He_2}"] & \mathcal{U}(\X{\He_2})
    \end{tikzcd}
    \hspace{2cm}
    \begin{tikzcd}
        {X_1} \arrow[r,"\epsilon_{X_1}"] \arrow[d,"f"] & \X{\mathcal{U}(X_1)} \arrow[d,"\mathbf{J}\circ\mathbf{G}(f)"] \\ 
        {X_2} \arrow[r,"\epsilon_{X_2}"] & \X{\mathcal{U}(X_2)}
    \end{tikzcd}
   \caption{Commutative diagrams for $\mathbf{G}$ and $\mathbf{J}$ where $\mathbf{M}_1=\langle\He_1,\square_1,\Diamond_1\rangle$, $\mathbf{H}_2=\langle\He_2,\square_2,\Diamond_2\rangle$, and $\mathcal{X}_1=\langle X_1,\tau_1,\leq_1,\eta_1,\eta_2\rangle$ and $\mathcal{X}_2=\langle X_2,\tau_2,\leq_2,\eta'_1,\eta'_2\rangle$.}
    \label{fig3}
\end{figure} 

Therefore, we have the following Theorem.

\begin{theorem}
    The categories $\mathbb{ME}$ and $\MH$ are dually equivalent.
\end{theorem}
\begin{proof}
    It follows from Theorems \ref{natis} and \ref{natep} that for all $\mathbf{M}=\langle\He,\square,\Diamond\rangle\in\MH$ and $\mathcal{X}=\langle X,\tau,\leq,\eta_1,\eta_2\rangle\in \mathbb{ME}$, $\sigma_\He\colon \He\to \mathcal{U}(\X{\He})$ and $\epsilon_X\colon X\to \X{\mathcal{U}(X)}$ are natural isomorphisms between $\mathbf{G}\circ\mathbf{J}$ and $\Id_{\MH}$ and $\mathbf{J}\circ\mathbf{G}$ and $\Id_{\mathbb{ME}}$. 
\end{proof}

\subsection{Topological duality for modal Nelson lattices}

We now introduce a category of topological structures that we will prove to be equivalent to the category $\mathbb{TW}$. Recall that filters of Heyting algebras correspond to closed up-sets in their dual Esakia space. To characterize boolean filters in terms of their dual space, we note that for all  $P \in \X{\He}$, the inclusion $D(\He) \subseteq P$ holds if and only if $P$ is maximal in the poset $\langle \X{\He}, \subseteq \rangle$.

So, let $\He$ be a Heyting algebra and let $F\subseteq H$ be a filter such that $D(\He)\subseteq F$. Then $\mathcal{C}(F)=\{P\in \X{\He}:F\subseteq P\}$ is a closed set such that $\mathcal{C}(F)\subseteq \max(\X{\He})$. Thus, following a similar approach to \cite{JanRivi2014} and \cite{Jansana2}, we define:

\begin{definition}
A modal NE-space ($MNE$-space) is a structure $\mathcal{X} = \langle X, \tau,\leq, \eta_\Box, \eta_\Diamond, C\rangle$ such that $\langle X, \tau, \leq,\eta_1,\eta_2 \rangle$ is an ME-space (according to Definition \ref{ME-space}) and $C\subseteq\max(X)$ is a closed subset satisfying the following property for all clopen up-sets $U, V \in \mathcal{U}(X)$:
\begin{equation}\label{F*}
    \text{If }C \subseteq U \cup V\text{ and }U \cap V = \emptyset\text{, then }C \subseteq \Box_{\eta_\Box} U \cup \Diamond_{\eta_\Diamond}V. \tag{$F^*$}
\end{equation}
\end{definition}

\begin{proposition}
    Let $\langle \He,\square,\Diamond\rangle$ be a modal Heyting algebra and let $F\subseteq H$ be a boolean filter of $\He$ that satisfies Condition $\eqref{F}$. Then, the structure $\langle \X{\He},\tau_\He,\subseteq,\eta_\square,\eta_\Diamond,\mathcal{C}(F)\rangle$ is a MNE-space.
\end{proposition}

\begin{proof}
    Since $\langle \He,\square,\Diamond\rangle$ is a modal Heyting algebra, we obtain that $\langle \X{\He},\tau_{\He},\subseteq,\eta_{\square},\eta_{\Diamond}\rangle$ is a modal Esakia space. Moreover, $F$ is a boolean filter, thus the closed set satisfies $\mathcal{C}(F)\subseteq \max(\X{\He})$. Suppose that $\mathcal{C}(F) \subseteq U \cup V$ and $U \cap V = \emptyset$ for $U,V\in\mathcal{U}(\X{\He})$. Then, there exist $a,b\in H$ such that $U=\sigma_{\He}(a)$ and $V=\sigma_{\He}(b)$. It follows that $\mathcal{C}(F)\subseteq \sigma_{\He}(a\vee b)$ and $\sigma_{\He}(a\wedge b)=\emptyset$. Thus, $a\vee b\in F$ and $a\wedge b=\bot$. By Condition \eqref{F}, $\square a\vee\Diamond b\in F$. Therefore, we obtain $\mathcal{C}(F)\subseteq \sigma_{\He}(\square a\vee \Diamond b)=\sigma_{\He}(\square a)\cup \sigma_{\He}( \Diamond b)=\square_{\eta_{\square}}(\sigma_{\He}(a))\cup \Diamond_{\eta_{\Diamond}}(\sigma_{\He}(a))=\square_{\eta_{\square}}(U)\cup \Diamond_{\eta_{\Diamond}}(V)$.
\end{proof}

\begin{corollary}\label{asMEN-space}
    Let $\A=\langle A,\blacksquare\rangle$ be a modal Nelson lattice. Then, the structure $\langle \X{\He_\mathbf{A}^*},\tau_{\He_{\mathbf{A}}^*},\subseteq,\eta_{\square^*},\eta_{\Diamond^*},\mathcal{C}(F^*)\rangle$ where $\mathcal{C}(F^*)=\{P\in \X{\He_\mathbf{A}^*}:F^*\subseteq P\}$ is a modal NE-space. 
\end{corollary}

\begin{proposition}
    Let $\mathcal{X} = \langle X, \tau, \leq,\eta_1, \eta_2, C\rangle$ be an MNE-space. Then, $F_\mathcal{C}=\{U\in\mathcal{U}(X):\mathcal{C}\subseteq U\}$ is a boolean filter of $\mathcal{U}(X)$ that satisfies Condition \eqref{F}. Thus, we obtain that the pair satisfies $(\langle\mathcal{U}(X),\square_{\eta_1},\Diamond_{\eta_2}\rangle,F_\mathcal{C})\in \mathbb{TW}$.
\end{proposition}
\begin{proof}
    Suppose that $U\in\mathcal{U}(X)$ is a dense element. Then, $X\setminus\da{U}=\emptyset$, i.e., $\da{U}=X$. We will show that $\mathcal{C}\subseteq U$. Let $x\in \mathcal{C}$. Then, there exists $u\in U$ such that $x\leq u$. Since $x\in\max(X)$, $x=u$ and thus $x\in U$. From Condition \ref{F*}, it is immediate that $F_\mathcal{C}$ satisfies Condition \eqref{F}.
\end{proof}


\begin{theorem}
    Let $\mathcal{X} = \langle X, \tau, \leq,\eta_1, \eta_2, C\rangle$ be an MNE-space. Let $\epsilon_X\colon X\to \X{\mathcal{U}(X)}$ be the map defined by $\epsilon_X(x)=\{U\in\mathcal{U}(X):x\in U\}$. Then, $\epsilon_X$ is an ME-homeomorphism that satisfies:
    \[\epsilon_X[C]=\mathcal{C}(F_C).\]
\end{theorem}
\begin{proof}
    Let $x\in C$. We will prove $F_C\subseteq \epsilon_X(x)$. Let $U\in F_C$. It follows that $x\in C\subseteq U$ and thus $U\in \epsilon_X(x)$. Therefore $\epsilon_X(x)\in \mathcal{C}(F_C)$.

    On the other hand, let $\epsilon_X(x)\in \mathcal{C}(F_C)$. Then, $F_C\subseteq \epsilon_X(x)$. To prove that $x\in C$, we will show that $x\in U$ for all $U\in\mathcal{U}(X)$ such that $C\subseteq U$. Let $C\subseteq U$. Then, $U\in F_C$ and by assumption, $U\in \epsilon_X(x)$, i.e., $x\in U$. Therefore $x\in C$ and $\epsilon_X(x)\in \epsilon_X[C]$.
\end{proof}

\begin{definition}
A map $f \colon X_1\to X_2$	 between two MNE-spaces $\langle X_1,\tau_1,\leq_1,\eta_1,\eta_2,C_1\rangle$ and $\langle X_2,\tau_2,\leq_2,\eta'_1,\eta'_2,C_2\rangle$ is an MNE-morphism iff is an ME-morphism that additionally satisﬁes: $f[C_1]\subseteq C_2$.

\end{definition}

\begin{proposition}
    Let $(\mathbf{M}_1,F_1)\in\mathbb{TW}$ and  $(\mathbf{M}_2,F_2)\in\mathbb{TW}$. Let $h\colon H_1\to H_2$ be a homomorphism of modal Heyting algebras that satisfies $h[F_1]\subseteq F_2$. If $f\colon \X{\He_2}\to\X{\He_1}$ is the function defined by $f(P)=h^{-1}[P]$, then $f$ is a MNE-morphism.
\end{proposition}
\begin{proof}
    We only need to prove that $f[\mathcal{C}(F_2)]\subseteq \mathcal{C}(F_1)$. Let $Q\in f[\mathcal{C}(F_2)]$. Then, there exists $P\in \mathcal{C}(F_2)$ such that $Q=h^{-1}[P]$. By definition, $F_2\subseteq P$, and by assumption, $h[F_1]\subseteq P$. Thus, $F_1\subseteq h^{-1}[P]=Q$ and $Q\in \mathcal{C}(F_1)$. Therefore, $f[\mathcal{C}(F_2)]\subseteq\mathcal{C}(F_1)$.
\end{proof}

\begin{proposition}
    Let $f \colon X_1\to X_2$ be a MNE-morphism	between two MNE-spaces $\langle X_1,\tau_1,\leq_1,\eta_1,\eta_2,C_1\rangle$ and $\langle X_2,\tau_2,\leq_2,\eta'_1,\eta'_2,C_2\rangle$. Then, the map $h\colon \mathcal{U}(X_2)\to\mathcal{U}(X_1)$ defined by $h(U)=f^{-1}[U]$ is a homomorphism of modal Heyting algebras that satisfies $h[F_{C_2}]\subseteq F_{C_1}$.
\end{proposition}

\begin{proof}
    Let $V\in h[F_{C_2}]$. Then, there exists $U\in \mathcal{U}(X_2)$ such that $C_2\subseteq U$ and $V=f^{-1}[U]$. By assumptio, $f[C_1]\subseteq C_2\subseteq U$ and it follows that $C_1\subseteq f^{-1}[U]=V$. Thus, $V\in F_{C_1}$.
\end{proof}

\begin{proposition}
    Let $(\mathbf{M},F)\in\mathbb{TW}$. Then, $\sigma_{\He}(a)=\{P\in\X{\He}:a\in P\}$ is a homomorphism of modal Heyting algebras that satisfies $\sigma_\He[F]=F_{\mathcal{C}(F)}$.
\end{proposition}

\begin{proof}
   We only need to prove $\sigma_\He[F]=F_{\mathcal{C}(F)}$. Let $U\in \sigma_\He[F]$. Then, there exists $a\in F$ such that $U=\sigma_\He(a)$. It follows that $\mathcal{C}(F)\subseteq\sigma_\He(a)$ and therefore $U=\sigma_\He(a)\in F_{\mathcal{C}(F)}$.
    On the other hand, let $U=\sigma_\He(a)\in F_{\mathcal{C}(F)}$. Then, $\mathcal{C}(F)\subseteq \sigma_\He(a)$. It follows that $a\in F$ and therefore $U\in \sigma_\He[F]$.
\end{proof}

Let $\mathbb{MNE}$ be the category of MNE-spaces with MNE-functions.
Let $\mathbf{L}\colon \mathbb{TW}\to \mathbb{MNE}$ be the contravariant functor that for each pair $P=(\mathbf{M},F)\in\mathbb{TW}$, $\mathbf{L}(P):=\langle \X{\He},\tau_\He,\subseteq,\eta_\square,\eta_\Diamond,\mathcal{C}(F)\rangle$ and that for each homomorphism of modal Heyting algebras $h\colon H_1\to H_2$ that satisfies $h[F_1]\subseteq F_2$, $\mathbf{L}(h)\colon \X{\He_2}\to\X{\He_1}$ is the MNE-morphism $\mathbf{L}(h):=X(h)$. 

Let $\mathbf{K}\colon \mathbb{MNE}\to \mathbb{TW}$ be the contravariant functor that for each MNE-space $\mathcal{X}=\langle X,\tau,\leq,\eta_1,\eta_2,\mathcal{C}\rangle$, $\mathbf{K}(\mathcal{X}):=(\langle \mathcal{U}(X),\square_{\eta_1},\Diamond_{\eta_2}\rangle,F_\mathcal{C})$ and that for each MNE-morphism $f\colon X_1\to X_2$, $\mathbf{K}(f)\colon \mathcal{U}(X_2)\to \mathcal{U}(X_1)$ is the modal Heyting algebra homomorphism $\mathbf{K}(f):=h(f)$. 
From all the theorems above, we can conclude:

\begin{theorem}
    The categories $\mathbb{MNE}$ and $\mathbb{TW}$ are dually equivalent.
\end{theorem}

And finally, we obtain a topological duality for modal Nelson lattices:
\begin{corollary}
    The category $\MN$ is dually equivalent to the category $\mathbb{MNE}$ via the functors $\mathbf{L} \circ\mathbf{F}\colon \mathbb{MN}\to \mathbb{MNE}$ and $\mathbf{E}\circ \mathbf{K}: \mathbb{MNE}\to \MN$.
\end{corollary}

\section{Conclusion and future work}\label{sec10}
In this work, we have extended the representation theorem for Nelson algebras in terms of Heyting algebras to the setting of modal operators. We have identified the minimal condition that a modal Heyting algebra must satisfy in order for the twist construction to be preserved in the modal setting. Moreover, we have shown how every equational extension on the Heyting side can be transferred to the Nelson side by means of $2$-potency. 

We believe that our results can be extended to a more general setting, since the approach is essentially independent of the underlying non-modal algebras. For instance, in \cite{Busaniche2}, the authors establish a representation theorem for a generalization of Nelson residuated lattices in terms of Nelson conucleus algebras. We believe that our results can be adapted to the setting of modal residuated lattices with relatively minor modifications. 

In addition, we would like to apply our twist contruction to a subvariety of Nelson lattices known as nilpotent minimum algebras (for short NM-algebras) which, in addition, satisfy the following equation:
\begin{equation*}
(a * b \to \bot) \vee (a \wedge b \to a * b) = \top 
\end{equation*}
The fact that nilpotent minimum algebras are Nelson residuated lattices has as a consequence that they are representable as a twist structure over G\"odel algebras (see Section 6.3 in \cite{Busaniche}), that are Heyting algebras that satisfy the prelinearity equation:
\[(x\rightarrow y)\vee (y\rightarrow x)=\top.
\]

Taking advantage of this representation, we want to reproduce our results in our current paper to modal NM-algebras and to study the connection with the caracterization given in \cite{CaiRod2015} of modal G\"odel algebras.

\backmatter

\bmhead{Acknowledgements} This research was funded by the National Science Center (Poland), grant number 2020/39/B/HS1/00216 and MOSAIC project from the European Union’s Horizon 2020 research and innovation programme under the Marie Skłodowska-Curie grant agreement No 101007627. The research was also partially supported by the Argentinean projects PIP 112-20200101301CO (CONICET) and PICT-2019-2019-00882 (ANPCyT). Finally, it acknowledges the partial support of Argentinean projects PIP 112-20150100412CO (CONICET) and UBA-CyT-20020190100021BA.

\end{document}